\documentclass[11pt,a4paper]{article}

\usepackage[T1]{fontenc}
\usepackage[utf8]{inputenc}
\usepackage{lmodern}
\usepackage{microtype}
\usepackage{geometry}
\usepackage{amsmath,amssymb,amsthm,mathtools}
\usepackage{enumitem}
\usepackage{xcolor}
\usepackage{tikz}
\usetikzlibrary{arrows.meta,positioning,calc,decorations.pathreplacing}
\usepackage{tikz-cd}
\usepackage{booktabs}
\usepackage{array}
\usepackage{hyperref}
\usepackage[capitalise,noabbrev]{cleveref}
\usepackage{orcidlink}
\usepackage{authblk}

\hypersetup{
  colorlinks=true,
  linkcolor=blue!50!black,
  citecolor=green!40!black,
  urlcolor=blue!60!black,
  pdftitle={AKSZ Descent on Manifolds with Ordinary Corners},
  pdfauthor={Cristian Anghel}
}

\newtheorem{theorem}{Theorem}[section]
\newtheorem{proposition}[theorem]{Proposition}
\newtheorem{lemma}[theorem]{Lemma}
\newtheorem{corollary}[theorem]{Corollary}
\newtheorem{conjecture}[theorem]{Conjecture}
\crefname{conjecture}{Conjecture}{Conjectures}
\Crefname{conjecture}{Conjecture}{Conjectures}
\theoremstyle{definition}
\newtheorem{definition}[theorem]{Definition}
\newtheorem{example}[theorem]{Example}
\newtheorem{assumption}[theorem]{Assumption}
\theoremstyle{remark}
\newtheorem{remark}[theorem]{Remark}

\newcommand{\F}{\mathcal F}
\newcommand{\Y}{\mathcal Y}
\newcommand{\Q}{\mathsf Q}
\newcommand{\ev}{\operatorname{ev}}
\newcommand{\Map}{\operatorname{Map}}
\newcommand{\Graph}{\operatorname{Graph}}
\newcommand{\Faces}{\mathfrak F}
\newcommand{\dd}{\mathrm d}
\newcommand{\cL}{\mathcal L}
\newcommand{\bfV}[1]{\mathrm{BF}^{#1}\mathrm V}

\title{AKSZ Descent on Manifolds with Ordinary Corners\\[2mm]
\large Face Complexes, Codimension-Two Compatibility, and Reduced Dirac Geometry}
\newcommand{\authorORCID}{0009-0005-2258-7935}
\author{Cristian Anghel\,\orcidlink{\authorORCID}}
\affil{Simion Stoilow Institute of Mathematics of the Romanian Academy,\\
21 Calea Grivi\c{t}ei, 010702 Bucharest, Romania\\
\texttt{cristian.anghel@imar.ro}\\
ORCID: \href{https://orcid.org/\authorORCID}{\texttt{\authorORCID}}}
\date{\today}

\begin{document}
\maketitle

\begin{abstract}
Under an explicit formal mapping-space hypothesis, we develop a facewise formulation of the classical Alexandrov--Kontsevich--Schwarz--Zaboronsky (AKSZ) construction on compact oriented manifolds with ordinary corners. The codimension-$r$ data---a mapping space carrying a closed two-form of degree $r-1$, an action of degree $r$, and a cohomological vector field---and the modified Batalin--Vilkovisky/Batalin--Fradkin--Vilkovisky Hamiltonian identity relating consecutive strata are those of the maximally extended BV--BFV theory of Cattaneo--Mnev--Reshetikhin. What is added here is the organization over the entire face poset: the Hamiltonian defect on a face is the sum of the pullbacks of the primitives on its codimension-one faces, weighted by the orientation incidence numbers, so that the boundary term of the single-stratum identity is resolved into its connected pieces with signs.

Organizing these defects by the face incidence complex yields a total-complex theorem: factorially normalized facewise transgression is a cochain map, so closed target forms transgress to cocycles, and the twice-iterated defect vanishes because the signed face differential squares to zero. We verify all four codimension-two cancellations explicitly for four-dimensional BF theory on $M=\Gamma\times[0,1]^2$.

We also establish a reduction criterion for singular corner data. If a raw codimension-two descendant is presymplectic and its reduced Dirac structure is the graph of a Poisson bivector, the shifted cotangent construction gives a canonical strict degree-two corner theory. The passage from a reduced Poisson bivector to a strict corner theory is already recorded in
\cite{CFT2026}; what is isolated here is the hypothesis under which it applies, and its relation to
the face-incidence structure. The construction provides a rigorous ordinary-corner benchmark for extensions of AKSZ descent to Joyce generalized corners.
\end{abstract}



\newpage

\tableofcontents

\section{Introduction}

The AKSZ construction associates a classical topological field theory to a differential graded source endowed with an invariant integration functional and to a Hamiltonian differential graded symplectic target \cite{AKSZ,CattaneoFelder}.  When the source is a closed oriented $n$-manifold $M$, one takes the differential graded source $T[1]M$ and transgresses the target symplectic form and Hamiltonian to the mapping space
\[
  \F_M=\Map(T[1]M,\Y).
\]
If the target symplectic form has degree $n-1$, the transgressed form has degree $-1$, and one obtains the usual BV theory.

For a source with boundary, integration by parts produces a boundary term.  The correct structure is therefore not an isolated BV manifold but a BV--BFV pair.  Cattaneo, Mnev, and Reshetikhin formulated the higher-codimension theory in \cite[Section~3.6, Definitions~3.26--3.27]{CMRClassical} and described the maximally extended AKSZ construction in \cite[Section~6.3]{CMRClassical}.
A codimension-$r$ stratum naturally carries a symplectic form of degree $r-1$ and an action of degree $r$.  Thus the familiar sequence
\[
  \mathrm{BV}\longrightarrow \mathrm{BFV}\longrightarrow
  \bfV{2}\longrightarrow\cdots
\]
is a degree-shifted form of iterated Stokes descent.

We isolate and prove a form of the construction adapted to two purposes: the passage to generalized corners, and comparison with reduced Dirac geometry.  The main organizational device is the \emph{face incidence complex}.  On an ordinary corner of depth two there are two orders in which one can approach the corner through boundary hypersurfaces.  Their induced orientations are opposite, and this is the geometric content of $\partial^2=0$.  We show that the same cancellation controls the iterated AKSZ defect, and we prove the resulting corner-square identity at the level of incidence numbers.

A second motivation comes from the recent work of Cattaneo, Fila-Robattino, and Tecchiolli on four-dimensional Palatini--Cartan gravity \cite{CFT2026}.  Their spacetime has only ordinary strata of codimension at most two.  Starting from the boundary constraint algebra, they obtain a pre-Dirac structure at the corner, reduce it to a genuine Dirac structure, identify the latter as the graph of a Poisson bivector, and then construct a strict $\bfV{2}$ theory.  This is importantly different from merely restricting a regular AKSZ theory: the raw gravitational corner structure is singular, and the strict theory appears only after reduction.

The contributions of this article are as follows.
\begin{enumerate}[label=(\roman*),leftmargin=2.2em]
\item We formulate facewise AKSZ transgression on a manifold with faces, with explicit incidence signs and degree bookkeeping.
\item We construct a face total complex and prove that factorially normalized transgression is a cochain map; the corner-square cancellation is its codimension-two component.
\item We restate the corner reduction criterion of \cite[Section~3.1.2]{CFT2026} for a
presymplectic codimension-two AKSZ descendant, and record the hypotheses under which the
strictification of \cite[Remark~3.14]{CFT2026} applies to it.
\item We explain exactly how the ordinary-corner geometry in \cite{CFT2026} fits this framework, while separating this geometric comparison from the unresolved problem of reconstructing Palatini--Cartan gravity by a single intrinsic AKSZ target.
\end{enumerate}

The restriction to ordinary corners is essential.  Locally, such corners are controlled by the free monoid $\mathbb N^r$ and their face lattice is Boolean.  Joyce's generalized corners replace $\mathbb N^r$ by an arbitrary weakly toric monoid $P$ and may have non-simplicial face combinatorics \cite{JoyceGCorner}.  The final section identifies the parts of the present construction that survive formally and the parts requiring genuinely new analysis.

\subsection*{Main theorem in compact form}

The principal statement can be summarized as follows.  It is stated under the standing formal
hypothesis \cref{ass:formal}, recorded in the paragraph \emph{Scope and analytic status}
below; its precise sign convention and proof appear in
\cref{thm:facewise,thm:ordered-transgression,thm:closed-face}.

\begin{theorem}[Coherent AKSZ descent over the face poset]\label{thm:intro-main}
Let $M$ be a compact oriented $n$-manifold with embedded ordinary faces and let
\[
 (\Y,\omega_{\Y}=\dd_{\Y}\alpha_{\Y},\Q_{\Y},\Theta)
\]
be an exact Hamiltonian dg symplectic target of degree $n-1$.  Under \cref{ass:formal}, the assignment
\[
 F\longmapsto
 (\F_F,\omega_F,\alpha_F,S_F,\Q_F)
\]
to all oriented faces $F\subseteq M$ has the following properties.
\begin{enumerate}[label=(\roman*),leftmargin=2.2em]
\item If $F$ has codimension $r$, then $\omega_F$ is closed with $|\omega_F|=r-1$, and $|S_F|=r$.
\item The failure of $\Q_F$ to be Hamiltonian is exactly the signed sum of the pullbacks of the primitives from the codimension-one faces of $F$.
\item Restriction to faces intertwines the cohomological vector fields.
\item Factorially normalized facewise transgression is a cochain map from the target de Rham complex to the face total complex.
\item The twice-iterated defect vanishes because the signed face differential squares to zero.
\item If $F$ is closed and $\omega_F$ is nondegenerate, the face data form a strict $\bfV{r}$ theory.
\end{enumerate}
\end{theorem}

The contribution is the explicit incidence-level organization, the face-poset coherence statement, and its use as a base case for comparison with Dirac reduction and generalized-corner geometry; the table below delineates it from the established literature.

\subsection*{Provenance and novelty audit}

\begin{center}
\small
\begin{tabular}{>{\raggedright\arraybackslash}p{0.29\textwidth} >{\raggedright\arraybackslash}p{0.62\textwidth}}
\toprule
Component & Status in this article\\
\midrule
BV--BFV extension to higher-codimension strata and AKSZ examples & Established in \cite{CMRClassical}; used as the formal-geometric foundation.\\
Facewise transgression formulas & Standard AKSZ Cartan calculus, rewritten with explicit incidence numbers and proved in the conventions of this paper.\\
Face total complex, normalized transgression cochain map, and corner-square identity & Developed here from standard AKSZ Cartan calculus; the global incidence organization and explicit face signs are the contribution, and the codimension-two compatibility is the transgressed form of $\partial_{\mathrm{face}}^2=0$.\\
Strictification from a reduced Poisson graph & Recorded in \cite[Remark~3.14]{CFT2026}; the
reduction hypotheses are those of \cite[Section~3.1.2 and Definition~3.13]{CFT2026}.  Restated
here for a transgressed AKSZ descendant, so as to be comparable with the face-incidence data.\\
Reduced Poisson corner of Palatini--Cartan gravity & Input from Cattaneo--Fila-Robattino--Tecchiolli \cite{CFT2026}, not reproved here.\\
AKSZ theory on Joyce generalized corners & Open problem; only obstructions and two candidate strategies are formulated.\\
\bottomrule
\end{tabular}
\end{center}

\subsection*{Relation to \texorpdfstring{\cite{CFT2026}}{the reduced Dirac structure paper}}

Because \cite{CFT2026} appeared while the present article was being completed, and because the
two overlap in more than one place, it is worth stating the relation explicitly rather than
leaving it to the table above.

Taken from \cite{CFT2026}: the degree conventions for $\bfV{k}$ data and their relaxed version
(their Definitions~2.16 and~2.18); the defect one-form $\check\alpha=\iota_{\Q}\omega-\delta S$
and the identity $\delta\check\alpha=-\cL_{\Q}\omega$ (their Remark~2.19), of which
\cref{cor:sympdefect} is the facewise refinement; the reduction of a pre-Dirac corner structure
and the notion of a theory good at a corner (their Section~3.1.2 and Definition~3.13), restated
here as \cref{ass:clean}; the passage from a reduced Poisson bivector to a strict $\bfV{2}$
theory (their Remark~3.14), which is \cref{prop:strictify}; the superfield description of the
four-dimensional $BF$ corner theory (their Section~3.2.2), recalled in \cref{sec:bf}; and, of
course, the reduced Poisson geometry of the Palatini--Cartan corner, which is the substantive
theorem of that paper and is not reproved here.  The characterization of Poisson graphs among
Dirac structures (\cref{lem:poisson-graph}) is classical \cite{Courant} and appears as their
Example~2.29; it is included only to keep the article self-contained.

Not in \cite{CFT2026}: the organization of the descent over the entire face poset, with the
single boundary pullback of the $k$-extended axiom resolved into a signed sum over connected
codimension-one faces; the total complex of \cref{thm:ordered-transgression} together with the
observation that the normalization $1/k!$ is forced by the coefficient $k$ in the
transgression--Stokes formula; the resulting corner-square identity; and the explicit
verification of all four codimension-two cancellations on $\Gamma\times[0,1]^2$
(\cref{prop:four-corners}).  The discussion of Joyce generalized corners in
\cref{sec:gcorners} is likewise independent of \cite{CFT2026}, which works throughout with
ordinary strata of codimension at most two.

\subsection*{Scope and analytic status}

All identities below are proved in the algebra of local differential forms on mapping spaces, in the sense of the standard AKSZ--BV--BFV Cartan calculus.  We make the following standing hypothesis.

\begin{assumption}[Formal mapping-space hypothesis]\label{ass:formal}
For every face $F$ the graded mapping space $\F_F=\Map(T[1]F,\Y)$ admits the evaluation map, restriction maps, contractions by the lifted source and target vector fields, and fiber integration of local forms.  These operations satisfy graded Cartan calculus and Stokes' formula.  In particular, all formulas may be read either:
\begin{enumerate}[label=(\alph*)]
\item formally, as identities between local functionals and local field-space forms; or
\item on a regular finite-dimensional approximation on which the indicated fiber integrals and pullbacks are defined.
\end{enumerate}
\end{assumption}

Finally, a bibliographic convention: the numbering of sections, definitions, propositions and
displayed equations quoted from \cite{CMRClassical} is that of the published version,
\emph{Commun.\ Math.\ Phys.} \textbf{332} (2014), which differs in places from the numbering of
the preprint arXiv:1201.0290.

The proofs are therefore complete at the classical formal-geometric level.  They do not assert that the full spaces of smooth fields carry globally defined Fr\'echet symplectic manifolds, nor that singular quotients occurring in gravity are smooth without the explicit regularity assumptions stated later.

\section{Ordinary corners and the face complex}

\subsection{Manifolds with faces}

We use a slightly restrictive but convenient class of manifolds with corners.

\begin{definition}\label{def:mwf}
A compact $n$-dimensional \emph{manifold with faces} is a compact smooth manifold with
corners $M$, in the sense of \cite{Melrose,JoyceCorners}, in which every boundary
hypersurface is embedded.
The closure of a connected stratum will be called a \emph{face}.  We write $\Faces_r(M)$ for the set of codimension-$r$ faces and $\Faces(M)=\coprod_r\Faces_r(M)$; if $M$ is connected then $\Faces_0(M)=\{M\}$, so that $M$ itself is a face of codimension zero.
Since $M$ is compact, each $\Faces_r(M)$ is finite, so all sums over faces below are finite.
\end{definition}

\begin{remark}\label{rem:depth}
Embeddedness implies that a point of depth $r$ lies in exactly $r$ distinct boundary
hypersurfaces: locally it lies on $r$ local boundary hypersurfaces, and embeddedness
prevents two of them from belonging to the same global one.  In particular a
codimension-two face is contained in exactly two codimension-one faces.

The same count applies inside every face, because a face of a manifold with faces is
again one.  Indeed, a face $F$ is a connected component of an intersection
$H_{i_1}\cap\cdots\cap H_{i_k}$ of distinct boundary hypersurfaces, hence a compact
manifold with corners, and its boundary hypersurfaces are the components of the
intersections $F\cap H_j$ with $j\notin\{i_1,\ldots,i_k\}$.  Were one of these
non-embedded in $F$, some $p\in F$ would lie on two local boundary hypersurfaces of $F$
contained in the same $F\cap H_j$, hence on two local boundary hypersurfaces of $M$
contained in the same $H_j$, which is excluded.  Consequently a codimension-two face of
$F$ lies in exactly two codimension-one faces of $F$; this is the count used in
\cref{lem:incidence}.
\end{remark}

This includes the local model
\[
 [0,\infty)^r\times\mathbb R^{n-r},
\]
and is sufficient for the field-theoretic applications considered here.  Every depth-$r$ point has normal monoid $\mathbb N^r$.

If $G$ is a codimension-one face of a face $F$, we write $G\prec F$.  Since $M$ is oriented and every face is co-orientable in $M$, every face is orientable; \emph{we fix once and for all an orientation of every face of $M$}, that of $M$ itself being the given one.  An orientation of $F$ induces an orientation of $G$ by the outward-normal-first convention.  Let
\[
 [F:G]\in\{+1,-1\}
\]
be the incidence number comparing this induced orientation with the fixed orientation of $G$.  The conventions are collected in \cref{app:signs}.

\begin{definition}\label{def:facechain}
The \emph{face chain complex} of $M$ is the free abelian group
\[
 C_r^{\mathrm{face}}(M)=\mathbb Z\langle \Faces_r(M)\rangle
\]
with boundary
\[
 \partial_{\mathrm{face}}F
 =\sum_{G\prec F}[F:G]G.
\]
\end{definition}

\begin{lemma}[Incidence cancellation]\label{lem:incidence}
For every codimension-two face $H$ of a face $F$,
\[
 \sum_{G:\,H\prec G\prec F}[F:G][G:H]=0.
\]
Consequently, $\partial_{\mathrm{face}}^2=0$.
\end{lemma}

\begin{proof}
By \cref{rem:depth} there are exactly two faces $G$ with $H\prec G\prec F$.  Put $m=\dim F$.  Near an interior point of $H$, choose oriented corner coordinates
\[
 (x_1,x_2,y_1,\ldots,y_{m-2})\in[0,\infty)^2\times\mathbb R^{m-2}
\]
on $F$, so that the two intermediate faces are $G_1=\{x_1=0\}$ and $G_2=\{x_2=0\}$.  Applying the outward-normal-first convention first along $x_1$ and then along $x_2$ gives the opposite orientation on $H$ from applying it in the reverse order; the computation is carried out in coordinates in \cref{app:cornersign}.  Thus
\[
 [F:G_1][G_1:H]=-[F:G_2][G_2:H]
\]
at that point.  Both sides are locally constant on $H$, and $H$ is connected because faces are closures of connected strata, so the identity holds on all of $H$ and the displayed sum vanishes.
\end{proof}

\subsection{Ordered corners}

For higher-codimension bookkeeping it is useful not to quotient immediately by permutations of the normal directions.

\begin{definition}\label{def:ordered-corner}
The \emph{ordered codimension-$r$ corner space} is the disjoint union
\[
 \partial^{[r]}M
 =\coprod_{(H_1,\ldots,H_r)} H_1\cap\cdots\cap H_r,
\]
where the union runs over ordered $r$-tuples of distinct boundary hypersurfaces whose intersection
is nonempty and clean of codimension exactly $r$; by \cref{rem:depth} this is automatic under
the embeddedness hypothesis of \cref{def:mwf}.  The symmetric group $S_r$ acts by permuting the order, and we write $\partial^{(r)}M=\partial^{[r]}M/S_r$ for the unordered corner space.
\end{definition}

Because the hypersurfaces in a tuple are distinct, the $S_r$-action is free and
\[
 \partial^{[r]}M\longrightarrow\partial^{(r)}M
\]
is a principal $S_r$-bundle.  The ordered version remembers precisely the orientation sign
lost by passing too early to the unordered corner.

The ordered corner space is not used in the constructions below: the total complex of
\cref{thm:ordered-transgression} is indexed by the face poset $\Faces(M)$, and every
cancellation is phrased through the incidence numbers $[F:G]$.  It is recorded here
because it is the geometric object that makes the sign in \cref{lem:incidence}
transparent, and because the passage to generalized corners in \cref{sec:gcorners} must
replace it rather than the incidence numbers themselves.

\begin{remark}[Comparison with the iterated boundary]\label{rem:iterated}
Melrose and Joyce define the corner spaces by iterating the boundary, a point of $\partial^rM$ being a point of $M$ together with an ordered choice of $r$ local boundary hypersurfaces through it \cite{Melrose,JoyceCorners,JoyceGCorner}.  Under the embeddedness hypothesis of \cref{def:mwf} every local boundary hypersurface determines a unique global one, and the two constructions agree: $\partial^rM\cong\partial^{[r]}M$ as manifolds with faces, compatibly with the $S_r$-actions.  We use the intersection description because the incidence numbers are read off from it directly.
\end{remark}

\begin{example}[The square]
For $M=[0,1]^2$, every vertex appears twice in $\partial^{[2]}M$: once for each order of the two incident edges.  The transposition in $S_2$ reverses the induced orientation sign.  This is the smallest geometric model for the corner-square identity proved below.
\end{example}

\subsection{Stokes' formula on the face complex}

Every face is compact, so no support condition is needed: for $F\in\Faces(M)$ and $\eta\in\Omega^{\dim F-1}(F)$, Stokes' formula on a manifold with corners \cite{Melrose} reads
\begin{equation}\label{eq:face-stokes}
 \int_F \dd_F\eta
 =\sum_{G\prec F}[F:G]\int_G\eta|_G,
\end{equation}
the sum being over the codimension-one faces of $F$ and $\eta|_G$ denoting the pullback along $G\hookrightarrow F$.  Applying \eqref{eq:face-stokes} twice produces no codimension-two contribution, by \cref{lem:incidence}.  The main point of this paper is that AKSZ transgression retains this face-complex structure on the infinite-dimensional mapping spaces.

\section{Hamiltonian dg targets and AKSZ transgression}\label{sec:transgression}

\subsection{Target data}

Let $n$ be the dimension of the bulk source.

\begin{definition}\label{def:target}
An \emph{exact Hamiltonian dg symplectic target of AKSZ degree $n-1$} is a quadruple
\[
 (\Y,\omega_{\Y}=\dd_{\Y}\alpha_{\Y},\Q_{\Y},\Theta)
\]
such that:
\begin{enumerate}[label=(\alph*)]
\item $\Y$ is a graded manifold;
\item $\omega_{\Y}$ is a symplectic two-form of internal degree $n-1$;
\item $\Q_{\Y}$ is a degree-$1$ cohomological vector field;
\item $\iota_{\Q_{\Y}}\omega_{\Y}=\dd_{\Y}\Theta$, where $\Theta$ has degree $n$;
\item $\{\Theta,\Theta\}_{\Y}=0$.
\end{enumerate}
\end{definition}

Given (d), condition (e) is equivalent to $\Q_{\Y}^2=0$: the vector field $\tfrac12[\Q_{\Y},\Q_{\Y}]$ is Hamiltonian with Hamiltonian $\tfrac12\{\Theta,\Theta\}_{\Y}$, and since $\omega_{\Y}$ is nondegenerate it vanishes precisely when $\{\Theta,\Theta\}_{\Y}$ is constant; as this function has degree $n+1\neq0$, the constant is zero.  Exactness is convenient because the BV--BFV boundary primitive is obtained by transgressing $\alpha_{\Y}$.

\subsection{Facewise mapping spaces}

Let $F$ be an oriented face of codimension $r$ and dimension $m=n-r$.  Set
\[
 \F_F:=\Map(T[1]F,\Y).
\]
The source de Rham vector field on $T[1]F$ will be denoted $D_F$.  It lifts by precomposition to a vector field $\widehat D_F$ on $\F_F$, and the target vector field lifts by postcomposition to $\check\Q_{\Y}$.  We define
\[
 \Q_F:=\widehat D_F+\check\Q_{\Y};
\]
that $\Q_F^2=0$ is proved in \cref{prop:Q-square} below.  Let
\[
 \ev_F:\F_F\times T[1]F\longrightarrow\Y,
 \qquad
 p_F:\F_F\times T[1]F\longrightarrow\F_F
\]
be evaluation and projection.  Finally, for an incidence $G\prec F$, restriction of maps defines
\[
 \rho_{FG}:\F_F\longrightarrow\F_G.
\]

\subsection{Transgression and Stokes calculus}

For a target differential form $\chi\in\Omega^p(\Y)$ define its ultra-local transgression by
\[
 \mathsf T_F^0\chi:=(p_F)_*\ev_F^*\chi,
\]
and for $k\geq 0$ set
\begin{equation}\label{eq:Tk}
 \mathsf T_F^k\chi:=\iota_{\widehat D_F}^{\,k}\mathsf T_F^0\chi.
\end{equation}
This is the local-form convention of \cite[Appendix D]{CMRClassical}, applied face by face.  The superscript $k$ records the number of source differentials inserted into the target form; it is not a power of the fiber-integration operator.

\begin{lemma}[Transgression--Stokes identity]\label{lem:transgression-stokes}
For every homogeneous $\chi\in\Omega^p(\Y)$ and every $k\geq 0$,
\begin{equation}\label{eq:transgression-stokes}
 \delta\bigl(\mathsf T_F^k\chi\bigr)
 =(-1)^m\left(
 \mathsf T_F^k(\dd_{\Y}\chi)
 -k\sum_{G\prec F}[F:G]\rho_{FG}^*\mathsf T_G^{k-1}\chi
 \right),
\end{equation}
where the boundary sum is absent for $k=0$.  Moreover,
\begin{align}
 \iota_{\widehat D_F}\mathsf T_F^k\chi&=\mathsf T_F^{k+1}\chi,\label{eq:insertD}\\
 \cL_{\widehat D_F}\mathsf T_F^k\chi
 &=(-1)^m\sum_{G\prec F}[F:G]\rho_{FG}^*\mathsf T_G^k\chi.\label{eq:LieD}
\end{align}
\end{lemma}

\begin{proof}
The first identity is proved in \cref{app:transgression}, by a self-contained computation in
superfields; for a source with a single boundary component it is the Cartan-calculus formula of
Cattaneo--Mnev--Reshetikhin \cite[Appendix~D, Eq.~(77)]{CMRClassical}, in the same convention for
\(
\mathsf T_F^k=\iota_{\widehat D_F}^{\,k}(p_F)_*\ev_F^*
\)
and for the sign contributed by integration over an $m$-dimensional source. For a manifold with faces, the restriction to the oriented boundary decomposes into its connected codimension-one faces. Hence the single boundary pullback in that formula becomes
\[
 \sum_{G\prec F}[F:G] \, \rho_{FG}^*\mathsf T_G^{k-1}\chi,
\]
by Stokes' formula \eqref{eq:face-stokes}; the factor $k$ is produced by the generating-function
argument of \cref{app:transgression} and agrees with the one in Eq.~(77).

Identity \eqref{eq:insertD} is the definition \eqref{eq:Tk}. Since $\widehat D_F$ has degree one, our Cartan convention is
\[
 \cL_{\widehat D_F}=\iota_{\widehat D_F}\delta-\delta\iota_{\widehat D_F}.
\]
Substituting \eqref{eq:transgression-stokes} for $k$ and $k+1$, the interior terms cancel and the two boundary coefficients combine as $(k+1)-k=1$. Projectability of the source de Rham vector field under restriction then yields \eqref{eq:LieD}; equivalently, it is the facewise version of \cite[Appendix~D, Eq.~(83)]{CMRClassical}.
\end{proof}

\begin{lemma}[Lifted target calculus]\label{lem:target-lift}
Let $V$ be a homogeneous vector field on $\Y$ and $\check V$ its pointwise lift to $\F_F$.  Then
\begin{align}
 \iota_{\check V}\mathsf T_F^k\chi
 &=(-1)^{(|V|+1)m}\mathsf T_F^k(\iota_V\chi),\label{eq:target-contraction}\\
 \cL_{\check V}\mathsf T_F^k\chi
 &=(-1)^{|V|m}\mathsf T_F^k(\cL_V\chi),\label{eq:target-Lie}\\
 [\widehat D_F,\check V]&=0.\label{eq:commuting-lifts}
\end{align}
\end{lemma}

\begin{proof}
Write $\mathsf T_F^k\chi=(p_F)_*\bigl(\iota_{\widehat D_F}^{\,k}\ev_F^*\chi\bigr)$.  Since $\check V$ is a pointwise lift, $\ev_F^*$ intertwines $\iota_V$ with $\iota_{\check V}$ and $\cL_V$ with $\cL_{\check V}$ before fiber integration.  It therefore suffices to commute the insertion past $(p_F)_*$, which lowers form degree by $m$ and therefore has parity $m$.  As operators on forms, $\iota_V$ has parity $|V|+1$ and $\cL_V$ has parity $|V|$; transposing them across $(p_F)_*$ costs $(-1)^{(|V|+1)m}$ and $(-1)^{|V|m}$ respectively.  These are the signs of \cref{app:signs}, and they are unaffected by the $k$ insertions of $\widehat D_F$, which commute with $\check V$ by the last identity.

For \eqref{eq:commuting-lifts}, note that $\widehat D_F$ acts by precomposition with the source de Rham vector field and $\check V$ by postcomposition with $V$.  Precomposition and postcomposition commute on a mapping space, and the graded commutator of the two lifts therefore vanishes irrespective of their parities.
\end{proof}

\subsection{Transgressed data and degrees}

With the notation above we set
\begin{align}
 \omega_F&:=\mathsf T_F^0(\omega_{\Y}),\label{eq:omegaF}\\
 \alpha_F&:=\mathsf T_F^0(\alpha_{\Y}),\label{eq:alphaF}\\
 S_F&:=\mathsf T_F^1(\alpha_{\Y})+\mathsf T_F^0(\Theta).
 \label{eq:SF}
\end{align}
By \cref{lem:transgression-stokes} with $k=0$ and $\omega_{\Y}=\dd_{\Y}\alpha_{\Y}$,
\begin{equation}\label{eq:exactF}
 \omega_F=(-1)^m\delta\alpha_F,
\end{equation}
where $\delta$ denotes the de Rham differential on field space.  In particular $\omega_F$ is closed.

\begin{proposition}[Degree bookkeeping]\label{prop:degrees}
If $F$ has codimension $r$ in $M$, then
\[
 |\omega_F|=|\alpha_F|=r-1,
 \qquad |S_F|=r,
 \qquad |\Q_F|=1.
\]
Thus $r=0$ gives BV degree, $r=1$ gives BFV degree, and $r=2$ gives $\bfV{2}$ degree.
These are the degrees prescribed for a $k$-extended BV theory in
\cite[Definition~3.26]{CMRClassical}; what is added here is that they are realized
facewise by transgression, with the incidence signs of \cref{def:facechain}.
\end{proposition}

\begin{proof}
The operator $\mathsf T_F^0$ lowers internal degree by $m=n-r$.  Therefore
\[
 |\omega_F|=(n-1)-m=r-1,
 \qquad
 |\alpha_F|=(n-1)-m=r-1.
\]
The insertion $\iota_{\widehat D_F}$ has degree $+1$ in ghost number, so
\[
 |\mathsf T_F^1\alpha_{\Y}|=(n-1)+1-m=r.
\]
Likewise $|\mathsf T_F^0\Theta|=n-m=r$.  Both summands of $S_F$ consequently have degree $r$.  The source and target cohomological vector fields both have degree one, hence so does their sum $\Q_F$.
\end{proof}

\begin{proposition}[Cohomologicality and projectability]\label{prop:Q-square}
For every face $F$ one has $\Q_F^2=0$.  If $G\prec F$, then
\begin{equation}\label{eq:Qproject}
 T\rho_{FG}(\Q_F)=\Q_G.
\end{equation}
\end{proposition}

\begin{proof}
Write $\Q_F=\widehat D_F+\check\Q_{\Y}$.  The source differential satisfies $D_F^2=0$, hence $[\widehat D_F,\widehat D_F]=0$.  The target vector field is cohomological, so $[\check\Q_{\Y},\check\Q_{\Y}]=0$.  The mixed commutator vanishes by \eqref{eq:commuting-lifts}; therefore $[\Q_F,\Q_F]=0$.

Restriction commutes with the de Rham differential on the source and with pointwise application of $\Q_{\Y}$.  Thus each summand of $\Q_F$ is $\rho_{FG}$-related to the corresponding summand of $\Q_G$.
\end{proof}

\section{The facewise AKSZ descent theorem}

\subsection{Modified Hamiltonian identity}

The following theorem is the facewise version of the classical AKSZ--BV--BFV identity.
Its local analytic content is standard: for a source with a single boundary component the
identity is the defining axiom of a $k$-extended BV theory,
$\iota_{Q_{N_{i-1}}}\omega_{N_{i-1}}=(-1)^{n-i+1}\delta S_{N_{i-1}}+\pi_i^*\alpha_{N_i}$,
in \cite[Definition~3.26]{CMRClassical}, and it is proved for AKSZ theories in
\cite[Proposition~6.3]{CMRClassical}.  The added point is the incidence organization over
all faces: the single boundary pullback is replaced by the signed sum
$\sum_{G\prec F}[F:G]\rho_{FG}^*$, which is what makes the codimension-two statements of
\cref{cor:cornersquare,thm:ordered-transgression} available.

\begin{theorem}[Facewise AKSZ descent]\label{thm:facewise}
Let $M$ be a compact oriented $n$-manifold with faces, and let $(\Y,\omega_{\Y}=\dd_{\Y}\alpha_{\Y},\Q_{\Y},\Theta)$ be an exact Hamiltonian dg symplectic target of AKSZ degree $n-1$ as in \cref{def:target}.  For every oriented face $F$ of dimension $m$, the data \eqref{eq:omegaF}--\eqref{eq:SF} satisfy
\begin{equation}\label{eq:main-descent}
 \iota_{\Q_F}\omega_F
 =(-1)^m\delta S_F
 +\sum_{G\prec F}[F:G]\,\rho_{FG}^*\alpha_G.
\end{equation}
Moreover, the restrictions are cohomological as in \eqref{eq:Qproject}.
\end{theorem}

\begin{proof}
We compute the source and target parts separately.

For the source part, use \eqref{eq:insertD}, exactness of the target symplectic form, and \eqref{eq:transgression-stokes} with $k=1$:
\begin{align*}
 \iota_{\widehat D_F}\omega_F
 &=\mathsf T_F^1(\omega_{\Y})
 =\mathsf T_F^1(\dd_{\Y}\alpha_{\Y})\\
 &=(-1)^m\delta\mathsf T_F^1(\alpha_{\Y})
 +\sum_{G\prec F}[F:G]\rho_{FG}^*\mathsf T_G^0(\alpha_{\Y})\\
 &=(-1)^m\delta S_F^{\mathrm{kin}}
 +\sum_{G\prec F}[F:G]\rho_{FG}^*\alpha_G.
\end{align*}
Here $S_F^{\mathrm{kin}}=\mathsf T_F^1\alpha_{\Y}$.

For the target part, $|\Q_{\Y}|=1$, so the sign in \eqref{eq:target-contraction} is trivial.  The target Hamiltonian identity and \eqref{eq:transgression-stokes} with $k=0$ give
\begin{align*}
 \iota_{\check\Q_{\Y}}\omega_F
 &=\mathsf T_F^0(\iota_{\Q_{\Y}}\omega_{\Y})
 =\mathsf T_F^0(\dd_{\Y}\Theta)\\
 &=(-1)^m\delta\mathsf T_F^0(\Theta)
 =(-1)^m\delta S_F^{\mathrm{int}}.
\end{align*}
Adding the two equalities proves \eqref{eq:main-descent}.  Cohomologicality and compatibility with restrictions were proved in \cref{prop:Q-square}.
\end{proof}

\begin{remark}[The facewise data as relaxed \texorpdfstring{$\bfV{r}$}{BFrV} manifolds]\label{rem:relaxed}
On a face $F$ with nonempty boundary, $\omega_F$ need not be nondegenerate and, by
\eqref{eq:main-descent}, $\Q_F$ is not Hamiltonian for $S_F$.  The quadruple
$(\F_F,\Q_F,\omega_F,S_F)$ is therefore not a $\bfV{r}$ manifold but a \emph{relaxed} one, in
the sense of \cite[Definition~2.18]{CFT2026}, and $(\dd_{\mathrm{face}}\alpha)_F$ is the
corresponding check one-form $\check\alpha$ of \cite[Remark~2.19]{CFT2026}.  The content of
\cref{thm:facewise} is that this defect is not merely some one-form on $\F_F$: it is the signed
sum of the transgressed primitives of the codimension-one faces of $F$, with the incidence
numbers of \cref{def:facechain}.
\end{remark}

\begin{remark}[Sign conventions]
Different BV--BFV references move the factor $(-1)^m$ between $\Q_F$, $S_F$, and $\alpha_F$.  The invariant content of \eqref{eq:main-descent} is that the failure of $\Q_F$ to be Hamiltonian is exactly the oriented sum of primitives on the adjacent faces.  All subsequent cancellation statements are independent of this global convention.
\end{remark}

Write $\dd_{\mathrm{face}}$ for the alternating pullback along codimension-one face inclusions: for a family $\eta=\{\eta_F\}_{F\in\Faces(M)}$ of field-space forms,
\begin{equation}\label{eq:dface}
 (\dd_{\mathrm{face}}\eta)_F:=\sum_{G\prec F}[F:G]\,\rho_{FG}^*\eta_G .
\end{equation}
This is the operator dual to the face boundary $\partial_{\mathrm{face}}$ of \cref{def:facechain}, and $\dd_{\mathrm{face}}^2=0$ follows from functoriality of restriction together with \cref{lem:incidence}.  With this notation \eqref{eq:main-descent} reads $\iota_{\Q_F}\omega_F-(-1)^m\delta S_F=(\dd_{\mathrm{face}}\alpha)_F$.

\begin{corollary}[Coherence data]\label{cor:coherence-data}
The family in \cref{thm:facewise} is coherent in the following explicit sense:
\begin{align*}
\rho_{FG}\circ\rho_{EF}&=\rho_{EG} &&(G\prec F\prec E),\\
T\rho_{FG}(\Q_F)&=\Q_G,\\
\delta\alpha_F&=(-1)^{\dim F}\omega_F,\\
\iota_{\Q_F}\omega_F-(-1)^{\dim F}\delta S_F
&=(\dd_{\mathrm{face}}\alpha)_F,\\
(\dd_{\mathrm{face}}^2\alpha)_H&=0.
\end{align*}
Thus the field restrictions, cohomological vector fields, primitives, symplectic forms, and actions are not independent collections: they form one descent object over the signed face category.
\end{corollary}

\begin{proof}
The first identity is functoriality of restriction; the second is \cref{prop:Q-square}; the third is \eqref{eq:exactF}; the fourth is \cref{thm:facewise}; and the fifth follows from \cref{lem:incidence}.  These identities are items (i)--(iii) of \cref{thm:intro-main}; items (iv)--(vi) are proved in
\cref{thm:ordered-transgression,thm:closed-face}.
\end{proof}

\subsection{Propagation of the symplectic defect}

Taking $\delta$ of \eqref{eq:main-descent} yields the next stage of descent.

\begin{corollary}\label{cor:sympdefect}
For every face $F$ of dimension $m$,
\begin{equation}\label{eq:sympdefect}
 \cL_{\Q_F}\omega_F
 =(-1)^{m}
 \sum_{G\prec F}[F:G]\rho_{FG}^*\omega_G.
\end{equation}
In particular, $\omega_F$ is $\Q_F$-invariant when $F$ is closed.
\end{corollary}

\begin{proof}
Because $|\Q_F|=1$, our Cartan convention is
\[
 \cL_{\Q_F}=\iota_{\Q_F}\delta-\delta\iota_{\Q_F}.
\]
Since $\delta\omega_F=0$ by \eqref{eq:exactF}, applying $\delta$ to \eqref{eq:main-descent} yields
\[
 -\cL_{\Q_F}\omega_F
 =\sum_{G\prec F}[F:G]\rho_{FG}^*(\delta\alpha_G).
\]
Now $\dim G=m-1$, and \eqref{eq:exactF} on $G$ gives
\[
 \delta\alpha_G=(-1)^{m-1}\omega_G.
\]
Consequently
\[
 \cL_{\Q_F}\omega_F
 =(-1)^m\sum_{G\prec F}[F:G]\rho_{FG}^*\omega_G,
\]
as claimed.
\end{proof}

Equation \eqref{eq:sympdefect} is the facewise form of the identity
$\check\omega=\delta\check\alpha=-\cL_{\Q}\omega$ for relaxed $\bfV{k}$ data, in the conventions of
\cite[Remark~2.19]{CFT2026}; its single-boundary form
$\cL_{Q_N}\omega_N=(-1)^{\dim N}\pi^*\omega_{\partial N}$ is
\cite[Proposition~D.2]{CMRClassical}, and the higher-codimension setting is
\cite[Section~3.6]{CMRClassical}.  It says that the non-invariance
of the degree-$(r-1)$ symplectic form on a codimension-$r$ face is measured by the degree-$r$
symplectic forms one step deeper.  This is the differential-geometric origin of the $\bfV{r}$ tower.

\subsection{The corner-square identity}

Let $H$ be a codimension-two face of $F$.  By \cref{rem:depth} there are exactly two intermediate faces $G_1,G_2$.

\begin{corollary}[Corner-square compatibility]\label{cor:cornersquare}
The total codimension-two contribution obtained by applying the AKSZ defect twice vanishes.  The
statement is the transgressed form of $\partial_{\mathrm{face}}^2=0$: its content is that the field
restrictions along the two orders agree, after which the incidence coefficients cancel by
\cref{lem:incidence}.  Explicitly,
\begin{equation}\label{eq:cornersquare}
 \sum_{G:\,H\prec G\prec F}
 [F:G][G:H]\,
 (\rho_{GH}\circ\rho_{FG})^*\alpha_H=0.
\end{equation}
Equivalently, the two ordered descents from $F$ to $H$ agree after correcting by the orientation sign induced by transposing the two normal directions.
\end{corollary}

\begin{proof}
Restriction is functorial, so both composites equal $\rho_{FH}$ as maps of fields.  The coefficient of $\rho_{FH}^*\alpha_H$ is
\[
 \sum_{G:\,H\prec G\prec F}[F:G][G:H],
\]
which vanishes by \cref{lem:incidence}.  The substance of the statement is thus the functoriality of restriction on fields, which makes the pullback independent of the intermediate face; the cancellation of the coefficients is then \cref{lem:incidence}.  In local coordinates, the two terms correspond to the ordered normals $(\partial_{x_1},\partial_{x_2})$ and $(\partial_{x_2},\partial_{x_1})$; transposition reverses the induced orientation.
\end{proof}

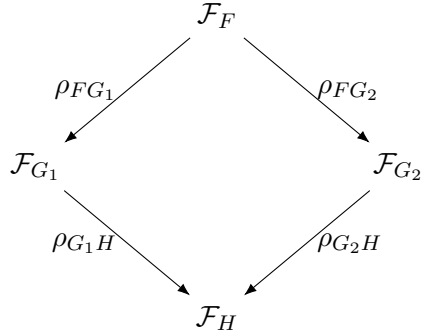
\begin{figure}[ht]
\centering
\begin{tikzpicture}[>=Latex]
\node (F)  at (0,0)     {$\F_F$};
\node (G1) at (-2.4,-2) {$\F_{G_1}$};
\node (G2) at (2.4,-2)  {$\F_{G_2}$};
\node (H)  at (0,-4)    {$\F_H$};
\draw[->] (F) -- node[left] {$\rho_{FG_1}$} (G1);
\draw[->] (F) -- node[right] {$\rho_{FG_2}$} (G2);
\draw[->] (G1) -- node[left] {$\rho_{G_1H}$} (H);
\draw[->] (G2) -- node[right] {$\rho_{G_2H}$} (H);
\node[below=8mm of H,align=center] {the square commutes on fields;\\the two incidence signs are opposite};
\end{tikzpicture}
\caption{The codimension-two restriction square.}
\label{fig:square}
\end{figure}

\begin{theorem}[Face transgression complex]\label{thm:ordered-transgression}
For $p\in\mathbb Z$, let $\mathcal C^p(M,\Y)$ consist of families
\[
 \eta=\{\eta_F^k\}_{F\in\Faces(M),\,k\geq0},
 \qquad
 \eta_F^k\in\Omega^{p-k}(\F_F),
\]
with the convention $\Omega^q(\F_F)=0$ for $q<0$. Put $\eta_F^{-1}=0$ and define
\begin{equation}\label{eq:ordered-total-differential}
 (\mathbf D\eta)_F^k
 :=(-1)^{\dim F}\delta\eta_F^k
 +\sum_{G\prec F}[F:G] \, \rho_{FG}^*\eta_G^{k-1}.
\end{equation}
Then $\mathbf D^2=0$.  Moreover, the factorially normalized facewise transgression
\begin{equation}\label{eq:normalized-transgression}
 \mathbf T(\chi)_F^k:=\frac1{k!}\mathsf T_F^k\chi,
 \qquad \chi\in\Omega^p(\Y),
\end{equation}
defines a cochain map
\[
 \mathbf T:(\Omega^\bullet(\Y),\dd_{\Y})
 \longrightarrow (\mathcal C^\bullet(M,\Y),\mathbf D).
\]
In particular, a closed target form transgresses to a $\mathbf D$-cocycle, while exactness of the target symplectic form gives
\[
 \mathbf D\mathbf T(\alpha_{\Y})=\mathbf T(\omega_{\Y}),
 \qquad
 \mathbf D\mathbf T(\omega_{\Y})=0.
\]
\end{theorem}

\begin{proof}
For a face $F$ of dimension $m$, the two mixed terms in $\mathbf D^2$ have coefficients $(-1)^m$ and $(-1)^{m-1}$, because every $G\prec F$ has dimension $m-1$; they cancel since pullback commutes with $\delta$.  The remaining horizontal term vanishes by functoriality of restriction and \cref{lem:incidence}.  Thus $\mathbf D^2=0$.

Multiplying \eqref{eq:transgression-stokes} by $(-1)^m/k!$ gives
\[
 (-1)^m\delta\!\left(\frac1{k!}\mathsf T_F^k\chi\right)
 =\frac1{k!}\mathsf T_F^k(\dd_{\Y}\chi)
 -\sum_{G\prec F}[F:G] \, \rho_{FG}^*
   \left(\frac1{(k-1)!}\mathsf T_G^{k-1}\chi\right),
\]
with the last term absent for $k=0$.  This is exactly
$\mathbf D\mathbf T(\chi)=\mathbf T(\dd_{\Y}\chi)$.  The final identities follow from
$\dd_{\Y}\alpha_{\Y}=\omega_{\Y}$ and $\dd_{\Y}\omega_{\Y}=0$.
\end{proof}

The theorem packages the complete hierarchy of facewise Stokes identities into a single cochain map.  Its horizontal square-zero relation is precisely the incidence mechanism behind \cref{cor:cornersquare}.  For generalized corners, it identifies the Boolean face complex of $\mathbb N^r$ as the combinatorial object that must eventually be replaced by the face complex of a weakly toric monoid.

\begin{theorem}[Closed-face master equation]\label{thm:closed-face}
Let $F$ be a closed codimension-$r$ face, $m=\dim F=n-r$.  Assume that $\omega_F$ is nondegenerate.  Then
\[
 \iota_{\Q_F}\omega_F=(-1)^m\delta S_F,
 \qquad \Q_F^2=0,
\]
and the Hamiltonian $(-1)^mS_F$ satisfies the classical master equation
\[
 \{S_F,S_F\}_F=0.
\]
Thus $(\F_F,\omega_F,\Q_F,S_F)$ is a strict $\bfV{r}$ theory, up to the harmless global sign convention for the Hamiltonian.
\end{theorem}

\begin{proof}
The boundary sum in \eqref{eq:main-descent} vanishes because $F$ is closed.  Nondegeneracy identifies $\Q_F$ with the Hamiltonian vector field of $(-1)^mS_F$.  For any symplectic form, the graded commutator of Hamiltonian vector fields is the Hamiltonian vector field of the Poisson bracket.  Hence
\[
 0=\tfrac12[\Q_F,\Q_F]
 =X_{\frac12\{(-1)^mS_F,(-1)^mS_F\}_F}.
\]
The bracket is therefore locally constant.  Its degree is $2r-(r-1)=r+1>0$, so no nonzero constant of that degree exists; it must vanish.  The overall sign cancels in the self-bracket.
\end{proof}

\begin{remark}[On the nondegeneracy hypothesis]\label{rem:nondeg}
Nondegeneracy of $\omega_F$ is not automatic and is not proved here in any
functional-analytic sense.  At the level of local field-space forms admitted by
\cref{ass:formal} it holds whenever $\omega_{\Y}$ is nondegenerate and $F$ is closed.
Indeed, $\mathsf T_F^0\omega_{\Y}=(p_F)_*\ev_F^*\omega_{\Y}$ is a Berezin fibre
integral over $T[1]F$, and Berezin integration retains only the top component in the
odd fibre coordinates.  It therefore pairs the part of a variation of source form
degree $p$ with the part of source form degree $m-p$, and the resulting pairing is
the composition of the pointwise nondegenerate pairing $\omega_{\Y}$ with the Poincar\'e
pairing $\Omega^p(F)\times\Omega^{m-p}(F)\to\mathbb R$, $(\eta,\eta')\mapsto\int_F\eta\wedge\eta'$.
The latter is nondegenerate on a closed oriented $F$, so $\omega_F$ is weakly
nondegenerate in the local-form sense; this is how it is used in \cref{prop:BF-cme}.
Weak nondegeneracy is all that the argument above requires, since it is what
identifies $\Q_F$ with a Hamiltonian vector field on local functionals.  Strong
nondegeneracy on a completed space of fields is a further analytic question, not
addressed here.
\end{remark}

\section{Codimension two: raw descent and strict \texorpdfstring{$\bfV{2}$}{BF2V} data}

\subsection{The regular AKSZ case}

Let $\Gamma$ be a codimension-two face of $M$.  Then
\[
 |\omega_\Gamma|=1,
 \qquad |S_\Gamma|=2.
\]
If $\Gamma$ is closed and $\omega_\Gamma$ is nondegenerate, \cref{thm:closed-face} with $r=2$ gives
\[
 \iota_{\Q_\Gamma}\omega_\Gamma
 =(-1)^{\dim\Gamma}\delta S_\Gamma,
 \qquad \{S_\Gamma,S_\Gamma\}_\Gamma=0,
\]
so that $(\F_\Gamma,\omega_\Gamma,\Q_\Gamma,S_\Gamma)$ is already a strict $\bfV{2}$ manifold.  Thus a regular AKSZ theory on an ordinary corner needs no additional construction at codimension two: the strict corner theory is already contained in transgression.

The difficulty in gravity is that the natural corner two-form can be degenerate and the relevant structure can be only pre-Dirac or homotopy Poisson.  One must then distinguish three operations:
\[
 \text{restriction},\qquad
 \text{BFV/AKSZ descent},\qquad
 \text{geometric reduction}.
\]
They need not commute without hypotheses.

\subsection{Dirac structures and Poisson graphs}

Let $P$ be a smooth (possibly infinite-dimensional) manifold.  The standard Courant bundle is
\[
 \mathbb T P=TP\oplus T^*P
\]
with its natural pairing and Dorfman bracket \cite{Courant}.  A Dirac structure $D\subset\mathbb TP$ is a maximally isotropic involutive subbundle.  If $\pi\in\Gamma(\wedge^2TP)$ is Poisson, then
\[
 \Graph(\pi)=\{\pi^\sharp(\eta)+\eta:\eta\in T^*P\}
\]
is a Dirac structure, and conversely a Dirac structure transverse to $TP$ is the graph of a
Poisson bivector; see \cite{Courant} and, in the conventions used here,
\cite[Example~2.29]{CFT2026}.

The graded encoding of a Poisson manifold is canonical \cite{Roytenberg}.  On $T^*[1]P$ let $\omega_{\mathrm{can}}$ be the degree-one symplectic form.  The bivector $\pi$ is a degree-two function
\[
 S_\pi=\frac12\pi^{ij}(x)p_ip_j,
\]
and
\[
 \{S_\pi,S_\pi\}_{\mathrm{can}}=0
 \quad\Longleftrightarrow\quad
 [\pi,\pi]_{\mathrm{Sch}}=0.
\]
Hence every Poisson manifold determines a strict $\bfV{2}$ manifold
\begin{equation}\label{eq:canonicalbf2v}
 \bigl(T^*[1]P,\omega_{\mathrm{can}},S_\pi,
 \Q_\pi=\{S_\pi,-\}\bigr).
\end{equation}
This is the observation recorded in \cite[Remark~3.14]{CFT2026}, and it is what
\cref{prop:strictify} below applies to a transgressed corner descendant.  In field-theoretic
examples, a pairing may identify the shifted cotangent presentation with a shifted tangent
presentation.

\begin{lemma}[Poisson graphs; {\cite{Courant}}, cf.\ {\cite[Example~2.29]{CFT2026}}]\label{lem:poisson-graph}
Let $P$ be a smooth manifold and $\pi\in\Gamma(\wedge^2TP)$.  The subbundle $\Graph(\pi)\subset TP\oplus T^*P$ is maximally isotropic.  It is involutive for the
Dorfman bracket if and only if $[\pi,\pi]_{\mathrm{Sch}}=0$.  The statement is classical; the proof
is included only because the maximality argument is the one point at which the infinite-dimensional
setting of \cref{ass:formal} requires comment.
\end{lemma}

\begin{proof}
For $\alpha,\beta\in T^*P$ the Courant pairing of $\pi^\sharp\alpha+\alpha$ and $\pi^\sharp\beta+\beta$ is
\[
 \beta(\pi^\sharp\alpha)+\alpha(\pi^\sharp\beta)
 =\pi(\alpha,\beta)+\pi(\beta,\alpha)=0.
\]
For maximality, suppose $X+\xi$ is orthogonal to every element of $\Graph(\pi)$.  Then for all $\eta\in T^*P$
\[
 \eta(X)+\xi(\pi^\sharp\eta)
 =\eta(X)-\eta(\pi^\sharp\xi)=0 ,
\]
using antisymmetry of $\pi$, whence $X=\pi^\sharp\xi$ and $X+\xi\in\Graph(\pi)$.  Thus $\Graph(\pi)^\perp=\Graph(\pi)$.  This argument uses no dimension count and therefore applies verbatim in the infinite-dimensional
setting of \cref{ass:formal}, provided the chosen functional-analytic pairing between $TP$ and
$T^*P$ is separating; this is what replaces the finite-dimensional rank count.  A direct calculation with the Dorfman bracket \cite{Courant,Roytenberg} gives
\[
 [\pi^\sharp\alpha+\alpha,\pi^\sharp\beta+\beta]_D
 =\pi^\sharp[\alpha,\beta]_\pi+[\alpha,\beta]_\pi
 +c\,\iota_{\alpha\wedge\beta}[\pi,\pi]_{\mathrm{Sch}},
\]
for a nonzero constant $c$ depending only on the normalisation of the Schouten
bracket, where
\[
 [\alpha,\beta]_\pi
 =\cL_{\pi^\sharp\alpha}\beta
 -\iota_{\pi^\sharp\beta}\dd_P\alpha.
\]
The first two terms lie in the graph.  The last term is a vector field with no
cotangent component, so it lies in the graph for all $\alpha,\beta$ exactly when it
vanishes; the value of $c$ is immaterial for this conclusion.  Since
$[\pi,\pi]_{\mathrm{Sch}}$ is a trivector field, its contractions with all pairs of
one-forms determine it, so this is equivalent to $[\pi,\pi]_{\mathrm{Sch}}=0$.
\end{proof}

\begin{lemma}[Shifted cotangent encoding]\label{lem:shifted-cotangent}
We use the Schouten convention
\[
 [\pi,\pi]_{\mathrm{Sch}}^{ijk}
 =2\bigl(\pi^{\ell i}\partial_\ell\pi^{jk}
 +\pi^{\ell j}\partial_\ell\pi^{ki}
 +\pi^{\ell k}\partial_\ell\pi^{ij}\bigr).
\]
On $T^*[1]P$ let $(x^i,p_i)$ be local coordinates of degrees $0$ and $1$.  For
\[
 S_\pi=\frac12\pi^{ij}(x)p_ip_j
\]
one has
\[
 \{S_\pi,S_\pi\}_{\mathrm{can}}
 =\frac16[\pi,\pi]_{\mathrm{Sch}}^{ijk}p_ip_jp_k.
\]
Consequently $S_\pi$ satisfies the classical master equation if and only if $\pi$ is Poisson.
\end{lemma}

\begin{proof}
Use the canonical degree-$(-1)$ bracket determined by $\{x^i,p_j\}=\delta^i_j$ and the graded Leibniz rule. Since $|S_\pi|=2$, the two contributions to its self-bracket are equal, and therefore
\[
 \{S_\pi,S_\pi\}_{\mathrm{can}}
 =2\left(\frac12\partial_\ell\pi^{ij}\,p_ip_j\right)
   \left(\pi^{\ell k}p_k\right)
 =\pi^{\ell k}\partial_\ell\pi^{ij}\,p_ip_jp_k.
\]
The product $p_ip_jp_k$ is totally antisymmetric. Averaging the coefficient over the three cyclic permutations gives
\[
 \pi^{\ell k}\partial_\ell\pi^{ij}\,p_ip_jp_k
 =\frac13\bigl(
 \pi^{\ell i}\partial_\ell\pi^{jk}
 +\pi^{\ell j}\partial_\ell\pi^{ki}
 +\pi^{\ell k}\partial_\ell\pi^{ij}
 \bigr)p_ip_jp_k.
\]
With the Schouten convention in the statement, the parenthesis is
$\tfrac12[\pi,\pi]_{\mathrm{Sch}}^{ijk}$, which proves the coefficient $1/6$.
\end{proof}

\subsection{A reduction-compatible descent criterion}

We now formulate the comparison principle used later for gravity.

\begin{assumption}[Clean corner reduction]\label{ass:clean}
The following packages, in the form convenient here, the reduction of a pre-Dirac corner
structure carried out in \cite[Section~3.1.2]{CFT2026}: the constraint submanifold, the
quotient by the kernel of the residual one-forms, and the requirement that the reduced
structure be the graph of a Poisson bivector are theirs, and a theory satisfying the
conclusion is called \emph{good on $\Gamma$} in \cite[Definition~3.13]{CFT2026}.  What is
specific to the present formulation is that the input is a transgressed AKSZ descendant, so
that the hypotheses can be compared with the face-incidence data of \cref{thm:facewise}.

Let
\[
 (\F_\Gamma^{\mathrm{pre}},\omega_\Gamma^{\mathrm{pre}},
  S_\Gamma^{\mathrm{pre}},\Q_\Gamma^{\mathrm{pre}})
\]
be a raw codimension-two descendant, where $\omega_\Gamma^{\mathrm{pre}}$ may be degenerate. Assume there exist:
\begin{itemize}[leftmargin=2.2em]
\item a $\Q_\Gamma^{\mathrm{pre}}$-invariant constraint submanifold $\mathcal C_\Gamma\subset\F_\Gamma^{\mathrm{pre}}$; and
\item a surjective submersion
\[
 q:\mathcal C_\Gamma\longrightarrow P_\Gamma.
\]
\end{itemize}
Assume further that:
\begin{enumerate}[label=(\alph*)]
\item the characteristic distribution of the presymplectic form $\omega_\Gamma^{\mathrm{pre}}$ has constant rank and is the vertical distribution of $q$;
\item the reduced Dirac structure is smooth and equals $\Graph(\pi_\Gamma)$ for a Poisson bivector $\pi_\Gamma$;
\item $\Q_\Gamma^{\mathrm{pre}}$ is projectable and the projectable Hamiltonian functionals descend to the Poisson algebra of $(P_\Gamma,\pi_\Gamma)$.
\end{enumerate}
\end{assumption}

\begin{proposition}[Strictification after Dirac reduction; cf.\ {\cite[Remark~3.14]{CFT2026}}]\label{prop:strictify}
Under \cref{ass:clean}, the reduced corner theory has a canonical strict $\bfV{2}$ completion given by
\[
 \F_\Gamma^{\mathrm{red}}=T^*[1]P_\Gamma,
 \qquad
 \omega_\Gamma^{\mathrm{red}}=\omega_{\mathrm{can}},
 \qquad
 S_\Gamma^{\mathrm{red}}=S_{\pi_\Gamma}.
\]
This completion is independent of the choice of representatives in $\mathcal C_\Gamma$.  If two clean reductions yield Poisson-isomorphic reduced manifolds, their strict $\bfV{2}$ completions are related by a graded symplectomorphism.
\end{proposition}

\begin{proof}
\cref{ass:clean}(a) says that functions and projectable tensors constant along the characteristic leaves descend uniquely through $q$.  \cref{ass:clean}(b), together with \cref{lem:poisson-graph}, implies that the descended bivector $\pi_\Gamma$ is Poisson.  By \cref{lem:shifted-cotangent}, the degree-two function $S_{\pi_\Gamma}$ on $T^*[1]P_\Gamma$ satisfies
\[
 \{S_{\pi_\Gamma},S_{\pi_\Gamma}\}_{\mathrm{can}}=0.
\]
Its Hamiltonian vector field is therefore cohomological, and the canonical degree-one form is symplectic.  This proves that the displayed data form a strict $\bfV{2}$ manifold.

The construction uses only the reduced tensor $\pi_\Gamma$, so changing a representative in a $q$-fiber changes none of the data.  If $\varphi:P_\Gamma\to P_\Gamma'$ is a Poisson diffeomorphism, its shifted cotangent lift
\[
 T^*[1]\varphi:T^*[1]P_\Gamma\longrightarrow T^*[1]P_\Gamma'
\]
preserves the canonical symplectic forms and carries $S_{\pi_\Gamma}$ to $S_{\pi_\Gamma'}$.  Hence it is an isomorphism of strict $\bfV{2}$ manifolds.
\end{proof}

\begin{remark}[Scope of the reduction criterion]\label{rem:reduction-scope}
The proposition concerns irreducible theories.  When the gauge symmetry is reducible, the reduced corner object is in general not a Poisson bivector but a higher structure $\Pi=\Pi_1+\Pi_2$ satisfying the Maurer--Cartan equation, with $\Pi_1$ cohomological and $\Pi_2$ a $\Pi_1$-invariant Poisson bivector; equivalently, a Dirac structure with support.  This already occurs for four-dimensional $BF$ theory, where the ghost-for-ghost field contributes; see \cite[Remarks~3.15 and~3.19--3.21]{CFT2026}.  Thus \cref{ass:clean}(b) should be read as excluding that case.

The proposition also concerns a single corner.  It does not assert that the strictified data are compatible with the face codifferential \eqref{eq:dface}, so the reduced corner theory is not claimed to sit inside the descent system of \cref{thm:facewise}.  Nor does it assert that every presymplectic AKSZ descendant admits a clean reduction, or that reduction automatically commutes with all BV pushforwards.  These are substantive analytic and geometric questions.  The statement isolates a sufficient condition under which a singular corner theory acquires a canonical strictification.
\end{remark}

\section{The ordinary-corner geometry of Palatini--Cartan gravity}

\subsection{The source geometry of the reduced Dirac construction}

Cattaneo, Fila-Robattino, and Tecchiolli \cite[\S2.1]{CFT2026} work with a stratified four-manifold
\[
 M=M^{[0]}\sqcup M^{[1]}\sqcup M^{[2]},
\]
whose points of codimension $k$ have neighborhoods modeled on
\[
 [0,\infty)^k\times\mathbb R^{4-k}.
\]
Thus their corners are ordinary corners.  In particular, near a connected codimension-two corner $\Gamma$ one has two independent normal directions and the local normal monoid is $\mathbb N^2$.  If $\Sigma$ is one adjacent boundary hypersurface, a collar is written
\[
 \Sigma_\varepsilon\simeq\Gamma\times[0,\varepsilon].
\]
This is exactly the local square underlying \cref{fig:square}.  No non-free monoid or Joyce generalized corner enters their construction.

\subsection{Why reduction precedes strictification}

The earlier BFV-induced corner construction for Palatini--Cartan gravity produces a singular
presymplectic structure and a local $P_\infty$ algebra rather than a strict $\bfV{2}$ manifold
\cite{CanepaCattaneoCorner}; we follow the summary of that construction given in
\cite[Section~1]{CFT2026}.  The obstruction to a direct BV--BFV implementation for the naive Palatini--Cartan--Holst action is established in \cite{CattaneoSchiavinaPCH}; see \cite{CattaneoGravityBoundary} for a later survey.  The construction of \cite{CFT2026} instead begins with the boundary constraint algebra, produces a pre-Dirac structure on the space of corner fields \cite[Section~3.1.2]{CFT2026}, and performs a geometric reduction.  On the reduced space $P_\Gamma$ it proves
\[
 D_\Gamma=\Graph(\pi_\Gamma),
\]
where $\pi_\Gamma$ is Poisson \cite[Theorem~5.29]{CFT2026}, the maximality of the reduced
Dirac structure following from the same statement, with the computation carried out in their
Appendix~B.2.  The strict $\bfV{2}$ theory is then the graded encoding of this reduced Poisson
geometry.

After field redefinitions, the paper obtains an affine-Poisson form of the corner action \cite[Proposition~6.6]{CFT2026},
\begin{equation}\label{eq:pcschematic}
 S^{\mathrm{PC}}_\Gamma
 =\int_\Gamma
 \left(
 \frac12 E[c,c]
 +\mu\,e\,\dd_{\omega_0}c
 +\frac12\mu^2F_{\omega_0}
 \right),
\end{equation}
with a degree-one symplectic structure pairing the reduced corner fields with their ghosts.  Here $e$ is the coframe, $\omega_0$ a reference connection with curvature $F_{\omega_0}$ and covariant derivative $\dd_{\omega_0}$, and $E=\tfrac12e^2$ is one of the two reduced corner variables, the other being $\Omega=e(\omega-\omega_0)$.  The field $c$ is the ghost of internal Lorentz transformations.  The field $\mu$ collects the remaining ghosts through
\[
 \mu=\iota_\zeta e+\lambda\epsilon_n+\eta\epsilon_m,
\]
where $\epsilon_n$ and $\epsilon_m$ are sections of the Minkowski bundle completing $e$ to a
local frame, $\epsilon_n$ being normal to $\Sigma$ and $\epsilon_m$ normal to $\Gamma$ inside
$\Sigma$.  Only $\epsilon_n$ is field-independent: $\epsilon_m$ is fixed by the normalization
conditions $[e,\epsilon_m]=[\epsilon_n,\epsilon_m]=0$ and $[\epsilon_m,\epsilon_m]=\pm1$, and
therefore depends on the coframe, so that $\delta\epsilon_m\neq0$
\cite[Definition~5.10 and Remark~5.13]{CFT2026}.  Here $\zeta$ generates diffeomorphisms tangent to $\Gamma$, $\eta$ the direction normal to $\Gamma$ inside $\Sigma$, and $\lambda$ the deformation normal to $\Sigma$.  Formula \eqref{eq:pcschematic} is BF-like but contains an affine curvature term determined by the background connection.  Moreover, \cite[Remark~6.7]{CFT2026} identifies it with the $\bfV{2}$ structure of four-dimensional $BF$ theory for $\mathfrak{so}(3,1)$ under the parametrization
\[
 A=\omega,\qquad B=E,\qquad \tau=e\mu,\qquad
 \phi=\tfrac12\mu^2,\qquad B^\dagger=0.
\]
Here $B^\dagger=0$ removes the degree-one component $\Pi_1$ of the $P_\infty$ structure discussed in
\cref{rem:reduction-scope}; the ghost-for-ghost $\phi$ itself is retained, with the nonzero value
$\phi=\tfrac12\mu^2$.  Compare \cite[Remark~3.20]{CFT2026}, where setting both $B^\dagger=0$ and
$\phi=0$ returns the purely Poisson corner structure.

In the language of \cref{ass:clean}, the substantive theorem of \cite{CFT2026} is the existence of a reduced manifold $P_\Gamma$ and a Poisson graph.  Once this has been established, \cref{prop:strictify} explains abstractly why a strict $\bfV{2}$ theory follows.

\subsection{What an AKSZ reconstruction must prove}

A successful AKSZ reconstruction of the gravitational corner theory must do more than observe that $\Gamma$ is an ordinary corner.  The construction of \cite{CFT2026} proceeds along a single boundary hypersurface and its infinitesimal collar, as summarized in \cite[Section~3.1.2]{CFT2026}.  An AKSZ reconstruction must reproduce the chain
\[
\begin{tikzcd}[column sep=2.1em,row sep=large]
 \F_M^{\mathrm{BV}} \arrow[r,"\mathrm{res}"]
 & \F_\Sigma^{\mathrm{BFV}} \arrow[r,"\mathrm{res}"]
 & \F_{\Sigma_\varepsilon} \arrow[r,"\mathrm{res}"]
 & \F_\Gamma^{\mathrm{pre}} \arrow[r,"\mathrm{red}"]
 & T^*[1]P_\Gamma
\end{tikzcd}
\]
with the following properties:
\begin{enumerate}[label=(\roman*),leftmargin=2.2em]
\item the boundary theory used as target is the correct reduced BFV theory of Palatini--Cartan gravity;
\item the raw corner descendant agrees with the pre-Dirac or pre-$\bfV{2}$ data before reduction;
\item the reduction is compatible with the cohomological vector field;
\item the construction compares consistently the two boundary hypersurfaces adjacent to $\Gamma$, with their ordered restrictions related by the incidence signs of \cref{cor:cornersquare};
\item the resulting Poisson bivector is the one computed in \cite{CFT2026}.
\end{enumerate}
The first and second points are obstructed for the naive Palatini--Cartan BV--BFV formulation;
this is why the ordinary-corner AKSZ theorem alone does not settle the gravitational problem.
The question is raised independently in \cite[Section~7]{CFT2026}, where the authors ask whether
there is a boundary BFV theory whose codimension-two descendant reproduces their strict
$\bfV{2}$ structure, and suggest that it might be related to the BFV theory of
\cite{CanepaCattaneoSchiavina} by a partial coisotropic reduction.  The list above is the form
that question takes once one insists, in addition, on compatibility with the face-incidence
structure of \cref{thm:facewise}.

This does not mean that the chain is unattainable, and the existing literature already
supplies a partial route.  Canepa and Cattaneo obtain the corner structure of
Palatini--Cartan gravity by a BV pushforward from the bulk description, and observe that
the same structure is reached from the corresponding BFV data by a construction of AKSZ
type \cite{CanepaCattaneoPushforward}; \cite{CFT2026} records both routes in its
introduction.  Items (i) and (ii) above are precisely the points at which that
observation would have to be upgraded: the pushforward produces the corner data, but it
does not exhibit a single intrinsic target $\Y$ on $\Gamma$ whose transgression is the
reduced theory, nor does it identify the raw descendant before reduction.  The
reconstruction problem is therefore not the existence of an AKSZ-type description, but
its intrinsic and facewise form, compatible with the incidence structure of
\cref{thm:facewise}.

\section{A model calculation: four-dimensional BF theory}\label{sec:bf}

The regular model against which gravity should be compared is four-dimensional BF theory.
On superfields we write $\dd_F$ for the source de Rham differential denoted $D_F$ in
\cref{sec:transgression}, and $\dd_{\mathcal A}=\dd_F+[\mathcal A,-]$ for the associated
covariant differential.  The
corner $\bfV{2}$ data recalled in this section are those of \cite[\S3.2.2]{CFT2026}, in the
superfield form used there; the contribution of the present section is the explicit orientation
bookkeeping of \cref{prop:four-corners} on $\Gamma\times[0,1]^2$.  Let $\mathfrak g$ be a finite-dimensional Lie algebra.  We fix in addition a nondegenerate invariant pairing on $\mathfrak g$; this is not needed for the target itself, but it lets us identify $\mathfrak g^*$ with $\mathfrak g$ and write both superfields with values in $\mathfrak g$.  The AKSZ target may be written as
\[
 \Y=T^*[3](\mathfrak g[1]),
\]
with canonical degree-three symplectic form.  In superfield notation the fields on an oriented face $F$ are
\[
 \mathcal A\in\Omega^\bullet(F,\mathfrak g)[1],
 \qquad
 \mathcal B\in\Omega^\bullet(F,\mathfrak g)[2].
\]
The transgressed structures are
\begin{align}
 \omega_F&=\int_F\langle\delta\mathcal B,\delta\mathcal A\rangle,
 \label{eq:bfomega}\\
 S_F&=\int_F\left\langle
 \mathcal B,\dd_F\mathcal A+\frac12[\mathcal A,\mathcal A]
 \right\rangle.
 \label{eq:bfaction}
\end{align}
By \cref{prop:degrees}, a face of codimension $r$ in a four-manifold carries $|\omega_F|=r-1$ and $|S_F|=r$.  For a codimension-two closed surface $\Gamma$,
\[
 |\omega_\Gamma|=1,
 \qquad
 |S_\Gamma|=2,
\]
and \eqref{eq:bfomega}--\eqref{eq:bfaction} give a strict $\bfV{2}$ theory without further reduction.  On a corner reached through two boundary hypersurfaces, the two ordered restrictions of $(\mathcal A,\mathcal B)$ coincide, while the orientation signs cancel exactly as in \cref{cor:cornersquare}.

\begin{remark}[Two routes to the BF corner]\label{rem:two-routes}
Four-dimensional BF theory is a reducible gauge theory, and \cref{rem:reduction-scope} noted that for such theories the classical Dirac reduction of the corner produces a higher structure rather than a Poisson bivector.  There is no conflict with the previous paragraph.  The two routes start from different data.  AKSZ transgression begins with an exact QP-target whose superfields already contain the ghost-for-ghost, so the strict $\bfV{2}$ structure is present from the outset and no reduction is performed.  The classical route of \cite{CFT2026} begins with the boundary constraint algebra, where reducibility appears as an obstruction to maximality of the induced Dirac structure and is resolved by a $P_\infty$ enhancement; see \cite[Remarks~3.15 and~3.19--3.21]{CFT2026}.  The comparison of the two descriptions for a reducible theory is not carried out here.
\end{remark}

\subsection{The explicit product corner \texorpdfstring{$\Gamma\times[0,1]^2$}{Gamma times the unit square}}

Let $\Gamma$ be a closed oriented surface and set
\[
 M=\Gamma\times[0,1]_t\times[0,1]_s,
 \qquad
 \operatorname{or}(M)=\dd t\wedge\dd s\wedge\operatorname{or}(\Gamma).
\]
Choose product orientations on the four hypersurfaces
\begin{align*}
 \Sigma_t^\varepsilon&=\Gamma\times\{\varepsilon\}\times[0,1]_s,
 &\operatorname{or}(\Sigma_t^\varepsilon)&=\dd s\wedge\operatorname{or}(\Gamma),\\
 \Sigma_s^\varepsilon&=\Gamma\times[0,1]_t\times\{\varepsilon\},
 &\operatorname{or}(\Sigma_s^\varepsilon)&=\dd t\wedge\operatorname{or}(\Gamma).
\end{align*}
We orient every corner by $\operatorname{or}(\Gamma_{\varepsilon\eta})=\operatorname{or}(\Gamma)$, where
\[
 \Gamma_{\varepsilon\eta}=\Gamma\times\{t=\varepsilon,\,s=\eta\};
\]
the incidence numbers $[\Sigma:\Gamma_{\varepsilon\eta}]$ computed below depend on this choice.  The base square is drawn in \cref{fig:product-corner}.  With these fixed orientations, the oriented boundary is
\begin{equation}\label{eq:product-boundary}
 \partial M=\Sigma_t^1-\Sigma_t^0-\Sigma_s^1+\Sigma_s^0.
\end{equation}
For any oriented face $F$ define
\[
 \alpha_F^{\mathrm{BF}}=\int_F\langle\mathcal B,\delta\mathcal A\rangle,
 \qquad
 \omega_F^{\mathrm{BF}}=\int_F\langle\delta\mathcal B,\delta\mathcal A\rangle,
\]
and let $S_F^{\mathrm{BF}}$ be \eqref{eq:bfaction}.  Here, and throughout this section, $\int_F$
denotes fibre integration $(p_F)_*$ over $T[1]F$, so that $\delta\circ(p_F)_*=(-1)^{\dim F}(p_F)_*\circ\delta$
as in \cref{app:signs}; with this reading $\omega_F^{\mathrm{BF}}=(-1)^{\dim F}\delta\alpha_F^{\mathrm{BF}}$,
in agreement with \eqref{eq:exactF} on faces of every parity.  The bulk identity is therefore
\begin{equation}\label{eq:BF-bulk-product}
 \iota_{\Q_M}\omega_M^{\mathrm{BF}}
 =\delta S_M^{\mathrm{BF}}
 +\rho_{M\Sigma_t^1}^*\alpha_{\Sigma_t^1}^{\mathrm{BF}}
 -\rho_{M\Sigma_t^0}^*\alpha_{\Sigma_t^0}^{\mathrm{BF}}
 -\rho_{M\Sigma_s^1}^*\alpha_{\Sigma_s^1}^{\mathrm{BF}}
 +\rho_{M\Sigma_s^0}^*\alpha_{\Sigma_s^0}^{\mathrm{BF}}.
\end{equation}
On each three-dimensional hypersurface the same calculation gives
\begin{equation}\label{eq:BF-face-product}
 \iota_{\Q_\Sigma}\omega_\Sigma^{\mathrm{BF}}
 =-\delta S_\Sigma^{\mathrm{BF}}
 +\sum_{\Gamma_{\varepsilon\eta}\prec\Sigma}
 [\Sigma:\Gamma_{\varepsilon\eta}]
 \rho_{\Sigma,\Gamma_{\varepsilon\eta}}^*
 \alpha_{\Gamma_{\varepsilon\eta}}^{\mathrm{BF}}.
\end{equation}
Finally, every $\Gamma_{\varepsilon\eta}$ is closed, so
\begin{equation}\label{eq:BF-corner-product}
 \iota_{\Q_{\Gamma_{\varepsilon\eta}}}
 \omega_{\Gamma_{\varepsilon\eta}}^{\mathrm{BF}}
 =\delta S_{\Gamma_{\varepsilon\eta}}^{\mathrm{BF}}.
\end{equation}

\begin{proposition}[Explicit cancellation at the four product corners]\label{prop:four-corners}
For each $\Gamma_{\varepsilon\eta}$, the two coefficients of the twice-iterated BF primitive are opposite.  With the product orientations above, the incidence products are
\[
\begin{array}{c|cc}
 & \text{path through }\Sigma_t^\varepsilon
 & \text{path through }\Sigma_s^\eta\\
\hline
\Gamma_{00} & (-1)(-1)=+1 & (+1)(-1)=-1\\
\Gamma_{01} & (-1)(+1)=-1 & (-1)(-1)=+1\\
\Gamma_{10} & (+1)(-1)=-1 & (+1)(+1)=+1\\
\Gamma_{11} & (+1)(+1)=+1 & (-1)(+1)=-1.
\end{array}
\]
Consequently the codimension-two contribution to the iterated defect of \eqref{eq:BF-bulk-product} is zero at every corner.
\end{proposition}

\begin{proof}
Equation \eqref{eq:product-boundary} gives the first factor in every table entry.  The second factor is obtained by taking the oriented boundary of the product-oriented interval in the corresponding hypersurface.  For example,
\[
 [M:\Sigma_t^0]=-1,
 \qquad
 [\Sigma_t^0:\Gamma_{00}]=-1,
\]
while
\[
 [M:\Sigma_s^0]=+1,
 \qquad
 [\Sigma_s^0:\Gamma_{00}]=-1.
\]
Thus the two paths to $\Gamma_{00}$ contribute $+\alpha_{\Gamma_{00}}^{\mathrm{BF}}$ and $-\alpha_{\Gamma_{00}}^{\mathrm{BF}}$.  The other rows follow identically.  The restrictions of $\mathcal A$ and $\mathcal B$ along the two paths are equal by functoriality, so the opposite coefficients cancel the same field-space one-form rather than merely two isomorphic copies.
\end{proof}

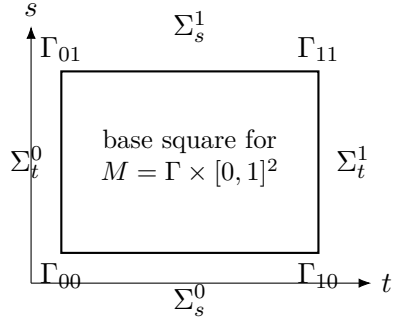
\begin{figure}[ht]
\centering
\begin{tikzpicture}[scale=1.0,>=Latex]
\draw[->] (0,0) -- (4.5,0) node[right] {$t$};
\draw[->] (0,0) -- (0,3.4) node[above] {$s$};
\draw[thick] (0.4,0.4) rectangle (3.8,2.8);
\node[below] at (0.4,0.4) {$\Gamma_{00}$};
\node[below] at (3.8,0.4) {$\Gamma_{10}$};
\node[above] at (0.4,2.8) {$\Gamma_{01}$};
\node[above] at (3.8,2.8) {$\Gamma_{11}$};
\node[anchor=east] at (0.30,1.6) {$\Sigma_t^0$};
\node[anchor=west] at (3.90,1.6) {$\Sigma_t^1$};
\node[below] at (2.1,0.12) {$\Sigma_s^0$};
\node[above] at (2.1,3.05) {$\Sigma_s^1$};
\node[align=center,font=\small] at (2.1,1.65) {base square for\\$M=\Gamma\times[0,1]^2$};
\end{tikzpicture}
\caption{The explicit codimension-two product model.  Each vertex is reached by two ordered boundary paths whose incidence products are opposite.}
\label{fig:product-corner}
\end{figure}

\begin{proposition}[Classical master equation for the BF corner]\label{prop:BF-cme}
Let $\Gamma$ be a compact oriented surface, possibly with boundary, and let $\mathfrak g$ carry an invariant nondegenerate pairing.  Write $F_{\mathcal A}=\dd_\Gamma\mathcal A+\tfrac12[\mathcal A,\mathcal A]$ for the superfield curvature.  The structures \eqref{eq:bfomega}--\eqref{eq:bfaction} are formally nondegenerate, in the sense that
$X\mapsto\iota_X\omega_\Gamma$ is injective on local vector fields, and if $\partial\Gamma=\varnothing$ they satisfy
\[
 \iota_{\Q_\Gamma}\omega_\Gamma=\delta S_\Gamma,
 \qquad
 \{S_\Gamma,S_\Gamma\}=0.
\]
If $\partial\Gamma\neq\varnothing$, the first identity acquires exactly the transgressed
primitive of \cref{thm:facewise},
\[
 \iota_{\Q_\Gamma}\omega_\Gamma
 =\delta S_\Gamma
 +\sum_{G\prec\Gamma}[\Gamma:G]\,\rho_{\Gamma G}^*\alpha_G^{\mathrm{BF}},
 \qquad
 \alpha_G^{\mathrm{BF}}=\int_G\langle\mathcal B,\delta\mathcal A\rangle .
\]
\end{proposition}

\begin{proof}
The pairing identifies variations of $\mathcal A$ and $\mathcal B$, so
\[
 \omega_\Gamma=\int_\Gamma\langle\delta\mathcal B,\delta\mathcal A\rangle
\]
is formally nondegenerate in the stated sense.

The Hamiltonian identity is the specialization of \cref{thm:facewise} to the target
$\Y=T^*[3](\mathfrak g[1])$ and the face $F=\Gamma$, for which $m=2$ and hence
$(-1)^m=+1$; the cohomological vector field is the one fixed in \cref{prop:Q-square}.
When $\partial\Gamma=\varnothing$ the incidence sum is empty and
$\iota_{\Q_\Gamma}\omega_\Gamma=\delta S_\Gamma$.  Since $\Q_\Gamma^2=0$ by
\cref{prop:Q-square}, the classical master equation then follows from
\cref{thm:closed-face}.

It is worth recording the same computation in variational form, since the boundary
sign is the one place where the graded conventions matter.  Using
$\delta\circ(p_\Gamma)_*=(p_\Gamma)_*\circ\delta$ for $m=2$, the Leibniz rule for the
even superfield $\mathcal B$, the anticommutation $\delta D=-D\delta$ of
\cref{app:signs}, and graded Stokes \eqref{eq:face-stokes}, one obtains
\begin{equation}\label{eq:BF-variation}
 \delta S_\Gamma
 =\int_\Gamma\Bigl(
 \langle\delta\mathcal B,F_{\mathcal A}\rangle
 -\langle\dd_{\mathcal A}\mathcal B,\delta\mathcal A\rangle
 \Bigr)
 -\int_{\partial\Gamma}\langle\mathcal B,\delta\mathcal A\rangle ,
\end{equation}
the bulk integral being $\iota_{\Q_\Gamma}\omega_\Gamma$.  The boundary term carries
the sign opposite to the one produced by a component-level integration by parts, in
which $\delta$ and $\dd$ commute; with the graded convention it reproduces the
incidence sum of \cref{thm:facewise}.  Since \eqref{eq:BF-variation} is obtained directly
from the superfield action, without invoking \cref{thm:facewise}, it is an independent
check of the incidence sign in \eqref{eq:main-descent} on a face with nonempty boundary,
and not merely a specialization of it.
\end{proof}

The quantization of this corner theory has been studied separately \cite{BFQuant2026}; we use only its classical layer.  This example clarifies the relation to the affine-Poisson gravitational formula.  The first term of \eqref{eq:pcschematic} is of Lie--Poisson type, while the terms containing $\omega_0$ and $F_{\omega_0}$ deform the linear BF corner bivector by an affine cocycle.  The gravity construction therefore resembles a constrained and affinely shifted BF corner theory, but the constraint reduction is essential and cannot be discarded.

\section{Gluing and locality at corners}

A central reason for using BV--BFV data is compatibility with cutting and gluing.  Suppose $M=M_1\cup_\Sigma M_2$ is obtained by gluing two manifolds with faces along oppositely oriented boundary hypersurfaces.  At the level of fields, the matching condition for local AKSZ fields is
\begin{equation}\label{eq:fiberproduct}
 \F_M\simeq \F_{M_1}\times_{\F_\Sigma}\F_{M_2}.
\end{equation}
For manifolds with corners this fiber product is the one studied by Kottke and Melrose
\cite{KottkeMelrose}.  At the level of field spaces, the statement that gluing corresponds
to a fiber product of spaces of fields---and, more precisely, to a fiber product taken up to
homotopy, with the homotopy constructed explicitly in a range of BV--BFV examples---is due to
Cattaneo and Mnev \cite{CattaneoMnevGluing}.  We do not reprove it: we take
\eqref{eq:fiberproduct} as a hypothesis rather than as a construction, and the content of this
section is only the behaviour of the incidence signs under it.  The boundary primitives contributed by $M_1$ and $M_2$ cancel because the induced orientations on $\Sigma$ are opposite.  At a codimension-two corner $\Gamma\subset\partial\Sigma$, there are two additional cancellations:
\begin{enumerate}[label=(\alph*)]
\item cancellation between the two bulk pieces along $\Sigma$;
\item incidence cancellation between the two ordered normal directions meeting at $\Gamma$.
\end{enumerate}

\begin{proposition}[Formal corner-locality]\label{prop:gluing}
Let $M=M_1\cup_\Sigma M_2$ with $\Sigma$ a common boundary hypersurface carrying opposite induced orientations from the two pieces, and assume \eqref{eq:fiberproduct}.  Then the modified Hamiltonian identities of the pieces glue to the identity for the glued manifold, and no uncancelled codimension-two term remains.
\end{proposition}

\begin{proof}
Write the descent identity for $M_1$ and $M_2$ and pull both identities to the fiber product of matching fields.  Additivity of fiber integration gives
\[
 \omega_M=\omega_{M_1}+\omega_{M_2},
 \qquad
 S_M=S_{M_1}+S_{M_2}.
\]
The hypersurface $\Sigma$ occurs in $\partial M_1$ and $\partial M_2$ with opposite induced orientations.  Hence its two primitive terms are
\[
 [M_1:\Sigma]\rho_1^*\alpha_\Sigma
 +[M_2:\Sigma]\rho_2^*\alpha_\Sigma=0
\]
on the matching locus.  Every external boundary face occurs once and therefore reproduces the boundary term for $M$.

Now let $\Gamma\subset\partial\Sigma$ be a codimension-two face.  Its coefficient in the iterated defect is
\[
 \sum_{G:\,\Gamma\prec G\prec M_i}[M_i:G][G:\Gamma].
\]
This coefficient vanishes separately on each piece by \cref{lem:incidence}; equivalently, the two ordered normal directions induce opposite orientations.  Thus no corner anomaly is created by adding the two descent identities.
\end{proof}

\begin{remark}
The conclusion can be stated more geometrically.  Near a point of $\Gamma=\partial\Sigma$ the two pieces contribute two quadrants glued along a half-line, so $\Gamma$ is not a corner of $M$ at all: it is a point of a boundary hypersurface of $M$.  There is therefore nothing left to cancel, and the vanishing of the two coefficients above is the algebraic shadow of this fact.  A matching condition on the two corner wedges is implicit in \eqref{eq:fiberproduct}.
\end{remark}

For reduced theories, this formal argument must be supplemented by compatibility of the reductions with fiber products.  This is one place where generalized corners may eventually be advantageous, because Joyce's category admits transverse fiber products under weaker combinatorial hypotheses than the category of ordinary corners \cite[\S4]{JoyceGCorner}.

\section{What has and has not been established}

The formalism is deliberately separated into proved statements, imported geometric input, and open extensions.

\subsection{Established under the standing hypothesis}

Within \cref{ass:formal}, this article proves:
\begin{enumerate}[label=(\roman*),leftmargin=2.2em]
\item the transgression--Stokes formula with face-incidence signs;
\item the facewise modified Hamiltonian identity in arbitrary codimension;
\item coherence of the restrictions and cohomological vector fields;
\item cancellation of the twice-iterated defect on every ordinary codimension-two corner;
\item strict $\bfV{r}$ data on closed nondegenerate faces;
\item that the shifted-cotangent strictification of \cite[Remark~3.14]{CFT2026} applies to a
transgressed codimension-two descendant once a clean Poisson reduction has been supplied, and is
then independent of the choice of representatives;
\item the explicit BF descent and all four orientation cancellations on $\Gamma\times[0,1]^2$.
\end{enumerate}

\subsection{Imported input}

The article does not rederive the boundary constraint algebra or the Dirac reduction of Palatini--Cartan gravity.  The existence of the reduced Poisson bivector $\pi_\Gamma$ is taken from \cite{CFT2026}.  Likewise, the existence of compatible BV--BFV models for gravity on cylinders and the relation to AKSZ-type constructions belong to the earlier literature \cite{CanepaCattaneoSchiavina,CanepaCattaneoPushforward}.

\subsection{Not claimed}

We do not claim:
\begin{enumerate}[label=(\roman*),leftmargin=2.2em]
\item a globally smooth Fr\'echet-manifold structure on every mapping space;
\item a theorem that singular gravitational reductions commute with restriction or BV pushforward;
\item a single finite-dimensional AKSZ target reconstructing four-dimensional Palatini--Cartan gravity on arbitrary cornered spacetimes;
\item a trace theorem for $b$-forms on Joyce generalized corners;
\item invariance of a generalized-corner construction under toric refinement.
\end{enumerate}
The last two items are proposed as central problems for the sequel.

\section{From ordinary to Joyce generalized corners: the next problem}\label{sec:gcorners}

This paper has used three special features of the local monoid $\mathbb N^r$.
\begin{enumerate}[label=(\roman*),leftmargin=2.2em]
\item The normal cone is simplicial and its codimension-one faces are generated independently.
\item Every codimension-two face is reached through exactly two orders of independent normal directions.
\item The face incidence complex is the Boolean complex, and its signs are the familiar alternating signs of a cube.
\end{enumerate}

For a Joyce generalized corner, the local model is
\[
 X_P=\operatorname{Hom}_{\mathrm{Mon}}(P,[0,\infty)),
\]
where $P$ is a weakly toric monoid.  When $P=\mathbb N^r\times\mathbb Z^{n-r}$ one recovers ordinary corners, but a general $P$ can have relations and a non-simplicial face lattice \cite{JoyceGCorner}.  The most elementary example is the positive real conifold
\[
 X_\square=\{(x_1,x_2,x_3,x_4)\in[0,\infty)^4:\ x_1x_2=x_3x_4\},
\]
which is three-dimensional and has four boundary hypersurfaces meeting at its vertex, rather than the three faces allowed by an ordinary $3$-corner.

The present construction suggests the following replacements.
\begin{center}
\begin{tabular}{>{\raggedright\arraybackslash}p{0.40\textwidth} >{\raggedright\arraybackslash}p{0.48\textwidth}}
\toprule
Ordinary corner & Generalized corner candidate\\
\midrule
free normal monoid $\mathbb N^r$ & weakly toric monoid $P$\\
Boolean face complex & oriented face complex of $P$\\
$T[1]X$ and ordinary de Rham differential & a dg source built from ${}^bT[1]X$ and the $b$-de Rham differential\\
ordinary Stokes integration & trace/renormalized integration with face residues\\
incidence cancellation of a cube & monoidal incidence identity in the face lattice of $P$\\
ordinary gluing fiber products & $b$-transverse fiber products in $\mathbf{Man}^{\mathbf{gc}}$\\
\bottomrule
\end{tabular}
\end{center}

The passage to generalized corners separates into three logically distinct obstructions.

\subsection{Combinatorial obstruction}

For $P=\mathbb N^r$, the interval of faces above a codimension-two face is a square: there are exactly two intermediate facets and the cancellation is pairwise.  For a general weakly toric monoid, the face lattice need not be Boolean.  The correct replacement of the Boolean cube is therefore an oriented incidence complex of faces of $P$.  Before defining any fields one must choose orientation lines and prove the monoidal analogue of
\[
 \sum_{G:\,H\prec G\prec F}[F:G][G:H]=0.
\]
The identity should follow from the cellular boundary of an oriented polyhedral cone, but its compatibility with Joyce coordinate changes has to be established.

\subsection{Differential-source obstruction}

The notation ${}^bT[1]X_P$ is a natural candidate because logarithmic tangent directions record the monoidal normal geometry.  It is not, by itself, a completed AKSZ source.  One must specify:
\begin{enumerate}[label=(\alph*),leftmargin=2.2em]
\item the graded algebra of functions to be identified with the chosen differential-form complex;
\item a degree-one homological vector field playing the role of the de Rham differential;
\item functorial restriction maps for all monoidal faces;
\item orientation data compatible with fiber integration.
\end{enumerate}
Moreover, using the $b$-de Rham complex is a choice rather than a formal consequence of Joyce's definition.  A competing construction may use ordinary forms on a resolution of $X_P$ and descend them after proving refinement invariance.

\subsection{Trace and Stokes obstruction}

If the $b$-de Rham route is chosen, local generators include logarithmic forms such as $\dd x/x$.  Their naive integrals can diverge at the boundary.  A generalized AKSZ trace must therefore be defined together with its domain and must satisfy a face-valued Stokes formula.  Possible mechanisms include compact-support conditions, residue maps, Hadamard finite parts, or pushforward from a resolved ordinary-corner model.  Until such a trace is constructed, the expression
\[
 \int_{{}^bT[1]X_P}\ev^*\omega_{\Y}
\]
is only symbolic.

\subsection{Two viable strategies}

The intrinsic strategy seeks a monoidal $b$-AKSZ source and a residue-compatible trace directly on $X_P$.  The resolution strategy chooses a toric refinement
\[
 \beta:\widetilde X\longrightarrow X_P
\]
for which $\widetilde X$ has ordinary corners, applies the theory of this article on $\widetilde X$, and then proves independence from the refinement.

\begin{conjecture}[Refinement invariance, provisional form]\label{conj:refinement}
Let $X$ be a compact oriented manifold with generalized corners and let $\Y$ be an exact AKSZ target.  Suppose two admissible toric refinements $\widetilde X_1\to X$ and $\widetilde X_2\to X$ admit a common refinement.  After imposing the appropriate support or renormalization conditions, the two ordinary-corner descent systems obtained on $\widetilde X_1$ and $\widetilde X_2$ are equivalent by BV--BFV pushforward over the common refinement.
\end{conjecture}

The conjecture is intentionally conditional: the notions of admissibility, equivalence, and pushforward must be fixed in the sequel.  Its advantage is that it converts the most difficult local geometry into a comparison problem for ordinary-corner theories, where \cref{thm:intro-main} is available.

A concrete program is therefore:
\begin{enumerate}[label=\arabic*.,leftmargin=2.2em]
\item define the oriented monoidal face complex and prove its incidence identity;
\item compare the intrinsic $b$-source and resolution approaches on affine models $X_P$;
\item construct a trace with a face-valued Stokes formula;
\item prove the analogue of \cref{thm:facewise};
\item prove or disprove \cref{conj:refinement} using blow-ups and generalized fiber products \cite{KottkeBlowup,KottkeMelrose};
\item recover the ordinary case $P=\mathbb N^2$ and compare it with the reduced gravitational corner theory.
\end{enumerate}

\section{Conclusions}

On an ordinary manifold with corners, the AKSZ construction is naturally a theory on the entire face poset, not merely a bulk theory supplemented by an external boundary correction.  The symplectic degree increases by one at each increase of codimension, and Stokes' formula becomes the modified Hamiltonian identity.  The equation $\partial_{\mathrm{face}}^2=0$ is the structural reason why iterated corner defects are compatible.

This organization separates two phenomena that can otherwise be conflated.  For regular topological models such as BF theory, facewise AKSZ transgression directly produces strict $\bfV{k}$ data.  For Palatini--Cartan gravity, the raw codimension-two geometry is singular; a strict $\bfV{2}$ theory appears only after the pre-Dirac structure is reduced to a Poisson graph.  The recent construction of \cite{CFT2026} supplies precisely this reduced Poisson geometry for an ordinary $\mathbb N^2$ corner.

The next step is not a formal replacement of $[0,\infty)^r$ by a more exotic local model.  It is the construction of a source differential and trace whose Stokes defect is controlled by the face lattice of a weakly toric monoid.  The quantum layer, in the sense of \cite{CMRQuantum}, is not addressed here at all.  Once this is achieved, the face-complex theorem proved here has a credible monoidal analogue and can serve as the ordinary-corner base case for an AKSZ--BV--BFV theory on Joyce generalized corners.

\appendix
\section{A compact sign dictionary}\label{app:signs}

Let $F$ be a face of dimension $m$.  With the convention
\[
 \alpha_F=\mathsf T_F^0(\alpha_{\Y}),
 \qquad
 \omega_F=\mathsf T_F^0(\omega_{\Y}),
\]
one has
\[
 \omega_F=(-1)^m\delta\alpha_F.
\]
The modified Hamiltonian identity is
\[
 \iota_{\Q_F}\omega_F
 =(-1)^m\delta S_F+\alpha_{\partial F},
 \qquad
 \alpha_{\partial F}:=
 \sum_{G\prec F}[F:G]\rho_{FG}^*\alpha_G.
\]
Consequently,
\[
 \cL_{\Q_F}\omega_F
 =(-1)^{m}\omega_{\partial F},
 \qquad
 \omega_{\partial F}:=
 \sum_{G\prec F}[F:G]\rho_{FG}^*\omega_G.
\]
The next boundary vanishes because
\[
 \sum_{H\prec G\prec F}[F:G][G:H]=0.
\]
A simultaneous redefinition of $\Q_F,S_F,$ and $\alpha_F$ can remove the factors $(-1)^m$; the face-incidence content is unchanged.

The transgression operators obey the following sign rules, used in \cref{lem:transgression-stokes,lem:target-lift}.  Fiber integration $(p_F)_*$ lowers form degree by $m$, so transposing an operator of parity $\pi$ across it costs $(-1)^{\pi m}$.  Applied to the field-space differential, which is odd, this gives
\[
 \delta\circ(p_F)_*=(-1)^m (p_F)_*\circ\delta ,
\]
which is the source of every factor $(-1)^m$ above.

The source de Rham vector field and the field-space differential are both odd, and
we use throughout the convention that they anticommute,
\[
 \delta\circ D_F=-D_F\circ\delta .
\]
This is compatible with the Cartan formula $\cL_V=[\iota_V,\delta]$ of
\cite[Appendix~D.4]{CMRClassical}, which for the degree-one field $\widehat D_F$ gives
$\cL_{\widehat D_F}=\iota_{\widehat D_F}\delta-\delta\iota_{\widehat D_F}$ as in
\cref{lem:transgression-stokes}.  It is the only point at which the graded formalism
differs from a naive component-level variational calculation, and it is responsible
for the sign of the boundary term in \cref{prop:BF-cme}.  For a homogeneous target vector field $V$ with pointwise lift $\check V$, the operator $\iota_V$ has parity $|V|+1$ and $\cL_V$ has parity $|V|$ on forms, whence
\[
 \iota_{\check V}\mathsf T_F^k\chi
 =(-1)^{(|V|+1)m}\mathsf T_F^k(\iota_V\chi),
 \qquad
 \cL_{\check V}\mathsf T_F^k\chi
 =(-1)^{|V|m}\mathsf T_F^k(\cL_V\chi).
\]
The full transgression--Stokes sign, including the coefficient $k$ of the boundary term, is fixed by the local Cartan-calculus identity \cite[Appendix~D, Eq.~(77)]{CMRClassical}; the facewise formula used here is obtained by decomposing the oriented boundary into its connected faces with incidence numbers.

\section{Local coordinate proof of the corner sign}\label{app:cornersign}

Let
\[
 F=[0,\infty)^2\times\mathbb R^{m-2}
\]
with orientation
\[
 \dd x_1\wedge\dd x_2\wedge\dd y_1\wedge\cdots\wedge\dd y_{m-2}.
\]
On $G_1=\{x_1=0\}$ the outward normal is $-\partial_{x_1}$, while on $G_2=\{x_2=0\}$ it is $-\partial_{x_2}$.  Contracting the orientation of $F$ with the outward normal gives the induced orientations
\[
 \iota_{-\partial_{x_1}}\operatorname{or}(F)
 =-\dd x_2\wedge\dd y_1\wedge\cdots\wedge\dd y_{m-2},
 \qquad
 \iota_{-\partial_{x_2}}\operatorname{or}(F)
 =+\dd x_1\wedge\dd y_1\wedge\cdots\wedge\dd y_{m-2}.
\]
Repeating the rule on each $G_i$ gives, on $H=\{x_1=x_2=0\}$,
\[
 \iota_{-\partial_{x_2}}\bigl(-\dd x_2\wedge\dd y_1\wedge\cdots\bigr)
 =+\dd y_1\wedge\cdots\wedge\dd y_{m-2},
 \qquad
 \iota_{-\partial_{x_1}}\bigl(+\dd x_1\wedge\dd y_1\wedge\cdots\bigr)
 =-\dd y_1\wedge\cdots\wedge\dd y_{m-2},
\]
so the two orders induce opposite orientations on $H$.  The asymmetry is produced at the
\emph{first} contraction, where $\iota_{\partial_{x_2}}$ crosses the factor $\dd x_1$ while
$\iota_{\partial_{x_1}}$ crosses nothing; neither of the second contractions crosses a factor.
Hence the two terms in \eqref{eq:cornersquare} cancel.  The same computation on $F=[0,1]$ with
$\operatorname{or}(F)=\dd x$ gives $[F:\{0\}]=-1$ and $[F:\{1\}]=+1$, so that
$\int_F\dd f=f(1)-f(0)$, which fixes the overall convention.

\section{Proof of the transgression--Stokes formula}\label{app:transgression}

We prove \eqref{eq:transgression-stokes} directly.  The argument is a generating-function
computation in superfields; it produces the coefficient $k$ of the boundary term and, at the same
time, shows that the normalization $1/k!$ of \eqref{eq:normalized-transgression} is forced rather
than chosen.

\subsection*{Superfield conventions}

Let $u^a$ be local coordinates on $\Y$, of internal degree $|a|$.  A point of
$\F_F=\Map(T[1]F,\Y)$ is a collection of superfields
\[
 \Phi^a\in C^\infty(T[1]F)=\Omega^\bullet(F),
 \qquad |\Phi^a|=|a| ,
\]
so that $\ev_F^*(u^a)=\Phi^a$.  We write $\dd$ for the source de Rham differential, acting on
superfields as the de Rham differential of $\Omega^\bullet(F)$; it is the operator induced by the
vector field $D_F$ on $T[1]F$.  We write $\delta$ for the field-space differential, and we use the
convention of \cref{app:signs},
\begin{equation}\label{eq:anticommute}
 \delta\circ\dd=-\dd\circ\delta ,
\end{equation}
so that the total operator
\[
 \mathsf d:=\delta+\dd
\]
satisfies $\mathsf d^2=0$.  Berezin integration
$\int_{T[1]F}=(p_F)_*$ has degree $-m$, whence
\begin{equation}\label{eq:intdelta}
 \int_{T[1]F}\delta\,\alpha=(-1)^m\,\delta\int_{T[1]F}\alpha ,
\end{equation}
while Stokes' formula \eqref{eq:face-stokes} on the face complex reads
\begin{equation}\label{eq:intd}
 \int_{T[1]F}\dd\,\alpha
 =\sum_{G\prec F}[F:G]\,\rho_{FG}^*\int_{T[1]G}\alpha ,
\end{equation}
the operator $\dd$ acting on the $\Omega^\bullet(F)$-coefficients only, so that no sign is
generated by the field-space form degree of $\alpha$.

For $\chi=\tfrac1{p!}\chi_{a_1\cdots a_p}(u)\,\dd u^{a_1}\cdots \dd u^{a_p}\in\Omega^p(\Y)$ the
ultra-local transgression is
\begin{equation}\label{eq:T0-explicit}
 \mathsf T_F^0\chi
 =\int_{T[1]F}\frac1{p!}\,\chi_{a_1\cdots a_p}(\Phi)\,
 \delta\Phi^{a_1}\cdots\delta\Phi^{a_p} .
\end{equation}

\subsection*{The insertion operator}

The vector field $\widehat D_F$ has internal degree $1$, so $\iota_{\widehat D_F}$ has total parity
$|\widehat D_F|+1=0$: it is an \emph{even} derivation of $\Omega^\bullet(\F_F)$ and obeys the
ungraded Leibniz rule.  On generators,
\[
 \iota_{\widehat D_F}\bigl(\delta\Phi^a\bigr)=\dd\Phi^a,
 \qquad
 \iota_{\widehat D_F}\bigl(f(\Phi)\bigr)=0,
 \qquad
 \iota_{\widehat D_F}\bigl(\dd\Phi^a\bigr)=0 .
\]
Consequently $\iota_{\widehat D_F}$ annihilates whatever it has already hit, and iterating it on a
product of $p$ generators gives
\begin{equation}\label{eq:iota-k}
 \iota_{\widehat D_F}^{\,k}\bigl(\delta\Phi^{a_1}\cdots\delta\Phi^{a_p}\bigr)
 =k!\sum_{\substack{S\subseteq\{1,\ldots,p\}\\ |S|=k}}
 \ \prod_{i=1}^{p}
 \begin{cases}\dd\Phi^{a_i}, & i\in S,\\[2pt] \delta\Phi^{a_i}, & i\notin S,\end{cases}
\end{equation}
the factors being kept in their original order.  The factor $k!$ counts the orders in which the $k$
chosen slots can be hit, and no signs occur because $\iota_{\widehat D_F}$ is even.

\subsection*{The generating function}

Define the \emph{total transgression} of $\chi$ by substituting $\Phi$ for $u$ and $\mathsf d\Phi$
for $\dd u$:
\begin{equation}\label{eq:total-transgression}
 \mathcal T_F(\chi)
 :=\int_{T[1]F}\frac1{p!}\,\chi_{a_1\cdots a_p}(\Phi)\,
 \mathsf d\Phi^{a_1}\cdots\mathsf d\Phi^{a_p} .
\end{equation}

\begin{lemma}\label{lem:generating}
$\displaystyle \mathcal T_F(\chi)=\sum_{k\geq0}\frac1{k!}\,\mathsf T_F^k\chi
 =\sum_{k\geq0}\mathbf T(\chi)_F^k .$
\end{lemma}

\begin{proof}
Expanding each factor $\mathsf d\Phi^{a_i}=\delta\Phi^{a_i}+\dd\Phi^{a_i}$ in
\eqref{eq:total-transgression} produces exactly one term for every subset
$S\subseteq\{1,\ldots,p\}$, namely the term in which the slots of $S$ carry $\dd\Phi$ and the
remaining ones carry $\delta\Phi$.  Grouping these terms by $|S|=k$ and comparing with
\eqref{eq:iota-k} and \eqref{eq:T0-explicit} gives
$\mathcal T_F(\chi)=\sum_k\frac1{k!}\iota_{\widehat D_F}^{\,k}\mathsf T_F^0\chi$, which is the
claim.  The sum is finite, since the $k$-th term vanishes for $k>p$.
\end{proof}

Thus the factorially normalized transgression is nothing but the homogeneous decomposition of a
single object: this is the conceptual reason for the coefficient $1/k!$ in
\eqref{eq:normalized-transgression}.

\subsection*{The identity}

\begin{proposition}\label{prop:transgression-stokes-proof}
For every homogeneous $\chi\in\Omega^p(\Y)$,
\begin{equation}\label{eq:generating-identity}
 (-1)^m\,\delta\,\mathcal T_F(\chi)
 =\mathcal T_F(\dd_{\Y}\chi)
 -\sum_{G\prec F}[F:G]\,\rho_{FG}^*\,\mathcal T_G(\chi),
\end{equation}
and \eqref{eq:transgression-stokes} is its component of field-space form degree $p-k+1$.
\end{proposition}

\begin{proof}
Since $\mathsf d=\delta+\dd$ is an odd differential with $\mathsf d^2=0$ by \eqref{eq:anticommute},
the substitution $u^a\mapsto\Phi^a$, $\dd u^a\mapsto\mathsf d\Phi^a$ is a morphism of differential
graded algebras.  Explicitly,
\[
 \mathsf d\left(\frac1{p!}\chi_{a_1\cdots a_p}(\Phi)\,
 \mathsf d\Phi^{a_1}\cdots\mathsf d\Phi^{a_p}\right)
 =\frac1{p!}\bigl(\partial_b\chi_{a_1\cdots a_p}\bigr)(\Phi)\,
 \mathsf d\Phi^{b}\,\mathsf d\Phi^{a_1}\cdots\mathsf d\Phi^{a_p},
\]
because $\mathsf d$ annihilates each $\mathsf d\Phi^{a_i}$; the right-hand side is the same
substitution applied to $\dd_{\Y}\chi$.  Hence, before integration,
\begin{equation}\label{eq:pullback-commutes}
 \mathsf d\bigl(\text{integrand of }\mathcal T_F(\chi)\bigr)
 =\text{integrand of }\mathcal T_F(\dd_{\Y}\chi).
\end{equation}
Now integrate \eqref{eq:pullback-commutes} over $T[1]F$ and split $\mathsf d=\delta+\dd$.  The
$\delta$-part contributes $(-1)^m\delta\,\mathcal T_F(\chi)$ by \eqref{eq:intdelta}.  The
$\dd$-part contributes $\sum_{G\prec F}[F:G]\rho_{FG}^*\mathcal T_G(\chi)$ by \eqref{eq:intd},
since restricting the superfields to $T[1]G$ is precisely $\rho_{FG}$ and the restricted integrand
is the integrand of $\mathcal T_G(\chi)$.  This is \eqref{eq:generating-identity}.

For the components, note that $\mathbf T(\chi)_F^k$ has field-space form degree $p-k$, that
$\mathbf T(\dd_{\Y}\chi)_F^k$ has degree $(p+1)-k$, and that $\mathbf T(\chi)_G^j$ has degree
$p-j$.  Matching degree $p-k+1$ in \eqref{eq:generating-identity} therefore pairs the $k$-th
component on the left with the $k$-th component of the first term and the $(k-1)$-st component of
the boundary sum:
\[
 (-1)^m\delta\left(\frac1{k!}\mathsf T_F^k\chi\right)
 =\frac1{k!}\mathsf T_F^k(\dd_{\Y}\chi)
 -\sum_{G\prec F}[F:G]\,\rho_{FG}^*\left(\frac1{(k-1)!}\mathsf T_G^{k-1}\chi\right),
\]
the last term being absent for $k=0$.  Multiplying by $(-1)^mk!$ gives exactly
\eqref{eq:transgression-stokes}, with the coefficient $k=k!/(k-1)!$ of the boundary term.
\end{proof}

\begin{remark}
Two features of the construction are visible in this proof.  First, the coefficient $k$ in
\eqref{eq:transgression-stokes} and the normalization $1/k!$ in
\eqref{eq:normalized-transgression} are the same phenomenon read in two gradings, which is why
\cref{thm:ordered-transgression} holds with no further adjustment.  Second, the only place where
the geometry of $F$ enters is \eqref{eq:intd}: the entire face-incidence structure of the theory
is imported through Stokes' formula on the face complex, and nowhere else.
\end{remark}

\section{The face transgression complex}\label{app:totalcomplex}

We record the sign check underlying \cref{thm:ordered-transgression}.  The complex is
\[
 \mathcal C^p(M,\Y)
 =\prod_{F\in\Faces(M)}\bigoplus_{k\geq0}
   \Omega^{p-k}(\F_F),
\]
and its differential is \eqref{eq:ordered-total-differential}.  If $m=\dim F$, then
\begin{align*}
 (\mathbf D^2\eta)_F^k
 &={(-1)^m}\delta
   \sum_{G\prec F}[F:G]\rho_{FG}^*\eta_G^{k-1}\\
 &\quad+
   \sum_{G\prec F}[F:G]\rho_{FG}^*
   \bigl((-1)^{m-1}\delta\eta_G^{k-1}\bigr)\\
 &\quad+
   \sum_{H\prec G\prec F}[F:G][G:H]
   (\rho_{GH}\circ\rho_{FG})^*\eta_H^{k-2}.
\end{align*}
The first two lines cancel because $\delta$ commutes with pullback.  In the final line the composite restriction is independent of the intermediate face, and its coefficient vanishes by \cref{lem:incidence}.  Hence $\mathbf D^2=0$.

For a target $p$-form $\chi$, the component $\mathbf T(\chi)_F^k$ of
\eqref{eq:normalized-transgression} has field-space form degree $p-k$, so the normalized family belongs to $\mathcal C^p(M,\Y)$.  The coefficient $1/k!$ is forced by the factor $k$ in the transgression--Stokes formula: it converts that term into the $(k-1)$-component of the horizontal differential.  Componentwise, \eqref{eq:transgression-stokes} is therefore equivalent to
\[
 \mathbf D\mathbf T(\chi)=\mathbf T(\dd_{\Y}\chi).
\]
This is the precise algebraic meaning of descent along the face poset.  In codimension two, the horizontal part of $\mathbf D^2=0$ is exactly the cancellation in \eqref{eq:cornersquare}.

\section*{Acknowledgements}

The author had numerous helpful discussions with Claude and ChatGPT during the preparation of this manuscript. These AI assistants were also used to help improve the language, edit the LaTeX source, and enhance the clarity and readability of the manuscript. The author take full responsibility for all scientific content, analyses, interpretations, conclusions, and any errors or omissions.

\end{document}